\documentclass[graybox]{svmult}

\usepackage{mathptmx}       
\usepackage{helvet}         
\usepackage{courier}        
\usepackage{type1cm}        
\usepackage{makeidx}         
\usepackage{graphicx}        
\usepackage{multicol}        
\usepackage[bottom]{footmisc}

\usepackage{amsmath,amssymb,latexsym}

\newcommand{\vrum}{\vspace{0.1cm}}

\newcommand{\masfR}{\mathsf R}

\newcommand{\aaf}{A_a^{\varphi,\phi}}

 \def\cN{\mathcal{N}}

\newcommand{\field}[1]{\mathbb{#1}}
\newcommand{\bR}{\field{R}}        
\newcommand{\bN}{\field{N}}        
\def\rdd{{\bR^{2d}}}
\def\R{\mathbb{R}}
\def\N{\mathbb{N}}
\def\sch{\mathcal{S}}
\def\Ren{\mathbb{R}^d}
\def\la{\langle}
\def\ra{\rangle}
\def\rd{\bR^d}

\def\cC{\mathcal{C}}

\def\Ren{\mathbb{R}^d}
\def\Renn{\mathbb{R}^{2d}}

\def\cD{\mathcal{D}}
\def\cE{\mathcal{E}}
\def\cS{\mathcal{S}}
\def\Nn{\mathbb{N}^{d}}
\def\Om{\Omega}
\def\o{\omega}
\def\cM{\mathcal{M}}

\def\inv{^{-1}}
\def\lrd{L^2(\rd)}

\def\Mmpq{M_m^{p,q}}
\def\phas{(x,\omega )}

\def\ni{\noindent}

\def\a{\alpha}                    

\newcommand{\tf}{time-frequency}

\newcommand{\stft}{short-time Fourier transform}
\def\Fur{\mathcal{F}}

\def\be{\begin{equation}}
\def\ee{\end{equation}}
\def\bena{\begin{eqnarray*}}
\def\ena{\end{eqnarray*}}

\def\mR{\mathbb{R}} \def\mN{\mathbb{N}}  \def\mC{\mathbb{C}}

\def\f{\varphi}

\def\gaw{A_a^{\f_1,\f_2}}

\def\phas{(x,\omega )}
\newcommand{\modsp}{modulation space}
\def\Nn{\mathbb{N}^{d}}
\def\N{\mathbb{N}}

\begin{document}

\title*{The Grossmann-Royer transform, Gelfand-Shilov spaces, and continuity properties of localization operators on modulation spaces}
\titlerunning{The Grossmann-Royer transform,...}

\author{Nenad Teofanov}

\institute{Nenad Teofanov \at Department of Mathematics and Informatics,
University of Novi Sad, Novi Sad, Serbia, \email{nenad.teofanov@dmi.uns.ac.rs}}

%
%

\maketitle

\abstract{
This paper offers a review of the results concerning  localization operators on modulation spaces, and related topics. However, our approach, based on the Grossmann-Royer transform, gives a new insight and (slightly) different proofs.
We define the Grossmann-Royer transform as interpretation of the  Grossmann-Royer operator in the weak sense.
Although such transform  is essentially the same as the cross-Wigner distribution,
the proofs of several known results are simplified when  it is used instead of other
time-frequency representations.
Due to the importance of their role in applications when dealing with ultrafast decay properties in phase space,
we give a detailed account on the Gelfand-Shilov spaces and their dual spaces,
and extend the  Grossmann-Royer transform and its properties in such context.
Another family of spaces, modulation spaces, are recognized as appropriate background
for time-frequency analysis. In particular, the Gelfand-Shilov spaces are revealed
as  projective and inductive limits of modulation spaces.
For the continuity and compactness properties of localization operators we employ the norms in modulation spaces.
We define localization operators
in terms of the Grossmann-Royer transform, and show that such definition coincides  with the
usual definition based on the short-time Fourier transform.
}

\section{Introduction}

The Grossmann-Royer operators originate from the problem of
physical interpretation of the Wigner distribution, see \cite{Grossmann, Royer}.
It is shown that the Wigner distribution is related to the expectation values of such operators
which describe reflections about  the phase space point.
Apart from this physically plausible interpretation of the Grossmann-Royer operators (sometimes called the parity operators),
they are closely related to the Heisenberg operators, the well known objects from quantum mechanics as  "quantized" variants of phase space translations. We refer to \cite{deGosson2011, deGosson2016, deGosson2017} for such relation, basic properties
of those operators, and the
definition of the cross-Wigner distribution $ W(f,g)$ in terms of  the Grossmann-Royer operator $R f $.

As it is observed in de Gosson \cite{deGosson2011} it is a pity that the Grossmann-Royer operators are not universally known.
One of the aims of this paper
is to promote those operators by defining the Grossmann-Royer transforms $R_g f$ as their weak interpretation.
The Grossmann-Royer transform is in fact the cross-Wigner distribution multiplied by $2^d$, see Definition \ref{transformacije}, Lemma  \ref{GRandRelatives}
and \cite[Definition 12]{deGosson2017}. Thus all the results valid for  $ W(f,g)$ can be formulated in terms of  $R_g f$.
Nevertheless, due to its physical meaning, we treat the Grossmann-Royer transform  as one of the time-frequency representations and perform a detailed exposition of its properties.
Our calculations confirm the remark from \cite[Chapter 8.3]{deGosson2011}
accorindg to which the use of $R_g f$  (instead of $V_g f, W(f,g)$ and $A(f,g),$ see  Definition \ref{transformacije})
"allows one to considerably simplify many proofs".

One aim of this paper is to give the interpretation of localization operators by the means of
the Grossmann-Royer transforms,  and to reveal its role in time-frequency analysis, but
we believe that, due to its intrinsic physical interpretation (as reflections in phase space),
the study of the Grossmann-Royer transform is of an  interest by itself.

To that end, we extend the definition of the Grossmann-Royer transform from the duality between the Schwartz space $ \mathcal{S} (\mR^d) $ and its dual space of tempered distributions $ \mathcal{S}' (\mR^d) $ to the whole range of Gelfand-Shilov spaces and their dual spaces. According to \cite{NR}, such spaces are "better adapted to the study of the problems of Applied Mathematics",
since they describe simultaneous estimates of exponential type for a function and its Fourier transform.
Here we give a novel (albeit expected)
characterization of Gelfand-Shilov spaces in terms of the Grossmann-Royer transform.

Gelfand-Shilov spaces can be described as the projective and the inductive limits of modulation spaces,
which are recognized as the right spaces for time-frequency analysis \cite{Gro}.
The modulation space norms traditionally  measure
the joint time-frequency distribution of $f\in \sch '$,
since their weighted versions are usually equipped with weights of at most polynomial growth at infinity (\cite{F1}, \cite[Ch.~11-13]{Gro}). The general approach to modulation spaces introduced already in  \cite{F1}
includes the weights of sub-exponential growth (see (\ref{BDweight})).
However, the study of ultra-distributions requires  the use of  the weights of exponential or even superexponential growth, cf. \cite{CPRT1, Toft-2017}. In our investigation of localization operators it is essential to use the (sharp) convolution estimates for modulation spaces, see Section \ref{Modulation Spaces}.

Localization operators are closely related to the above mentioned cross-Wigner distribution. They were first introduced by Berezin in the study of general Hamiltonians satisfying the so-called Feynman inequality, within a
quantization problem in quantum mechanics, \cite{Berezin71, Shif}.
Such operators and their modifications are also called Toeplitz or Berezin-Toeplitz operators,
anti-Wick operators and Gabor multipliers, see \cite{folland89,Toft2007b,Toft-2012}.
We do not intend to discuss the terminology here, and  refer to, e.g.
\cite{Englis} for the relation between Toeplitz operators and localization operators.

\par

In signal analysis, different localization techniques are used to describe signals which are as concentrated as possible in general regions of the phase space.
In particular, localization operators are related to localization technique developed by Slepian-Polak-Landau,
where  time and frequency are considered to be two separate spaces, see e.g. the survey article
\cite{Slepian}. A different construction is proposed by Daubechies in \cite{Daube88},
where time-frequency plane  is treated as one geometric whole (phase space).
That fundamental contribution which contains localization in phase space together with
basic facts on localization operators with references to applications in  optics and signal analysis,
initiated further study of the topic.
More precisely, Daubechies studied
localization operators $A_a ^{\varphi_1, \varphi_2} $ with Gaussian windows
$$\varphi_1 (x) = \varphi_2 (x) =  \pi^{-d/4} \text{exp} (-x^2/2), \;\;\; x\in \mathbb{R}^d, \;\;\;
\text{and with a radial symbol} \;\;\; a \in L^1 (\mathbb{R}^{2d} ).
$$
Such operators are named Daubechies' operators afterwards.
The eigenfunctions  of Daubechies' operators are $d-$dimensional Hermite functions
\begin{equation} \label{herm1}
H_k (x) = H_{k_1} (x_1)  H_{k_2} (x_2) \dots  H_{k_d} (x_d) = \prod_{j=1} ^d H_{k_j} (x_j), \;\;\;
x \in\bR^d, \;\;\; k \in\bN^d _0,
\end{equation}
and
$$
H_n (t) =  (-1)^n \pi ^{-1/4} (2^n n!)^{-1/2} \text{exp} (t^2/2) (\text{exp} (-t^2))^{(n)},
\;\;\; t \in \mathbb{R}, \;\;\; n = 0,1,\dots,
$$
and corresponding eigenvalues can be explicitly calculated.
Inverse problem for a simply connected localization domain $ \Omega $ is studied in
\cite{AbreuDorfler12} where it is proved that if one of the eigenfunctions of Daubechies' operator
is a Hermite function, then $ \Omega $ is a disc centered at the origin.
Moreover, the Hermite functions belong to
test function spaces for ultra-distributions, both in non-quasianalytic and in quasianalytic case,
and give rise to important representation theorems, \cite{L06}.
In our analysis this fact is used in Theorem \ref{kernelteorema}.

The eigenvalues and eigenfunctions in signal analysis are discussed in \cite{RamTop93},
where the localization operators of the form
$  \langle L_{\chi_\Omega} f, g \rangle = \iint_\Omega W(f,g), $ were observed.
Here $\Omega $ is a quite general open and bounded subset in $ \mathbb{R}^2 $,
$\chi _\Omega $ is its characteristic function, and
$ W(f,g)$ is the cross-Wigner transform (see Definition \ref{transformacije}).
In \cite{RamTop93} it is proved that the eigenfunctions of $ L_{\chi_\Omega} $ belong to the Gelfand-Shilov space
$ {\mathcal S}^{( 1)} (\mathbb{R}^d)$, see Section \ref{Gelfand} for the definition.

\par

Apart from applications in signal analysis,
the study of localization operators is motivated (and in fact initiated by Berezin)
by the problem of quantization.
An exposition of different quantization theories and connection between localization operators and Toeplitz operators is
given in e.g. \cite{Englis, Englis02}.
The quantization problem is in the background of the study of localization operators on
$ L^p (G)$, $ 1\leq p \leq \infty $, where $G$ is a locally compact group,
see \cite{BOW, Wong02}, where one can find product formulae and Schatten-von Neumann properties of localization operators in that context.

In the context of time-frequency analysis,  the groundbreaking contribution is given by Cordero and Gr\"ochenig, \cite{CG02}. Among other things, their results emphasize the role of modulation spaces in the study of localization operators.
Since then, localization operators in the context of modulation spaces and Wiener-amalgam spaces were studied
by many authors, cf. \cite{FeiNowak2003, CorGro2006, Toft2007a, Toft2007b, Weisz}.
See also the references given there.

In this paper we rewrite the definition of localization operators and Weyl pseudodifferential operators in terms of the
Grossmann-Royer transform and recover the main relations between them. Thereafter we
give an exposition of results concerning the continuity of localization operators on modulation spaces with polynomial weights, and their compactness properties both for polynomial and exponential weights.

\subsection{The content of the paper}

This paper is mainly a review of the results on the topic of localization operators on modulation spaces with some new insights and (slightly) different proofs.
More precisely, we propose a novel approach  based on the Grossmann-Royer transform $R_g f$,
defined in  Section \ref{Grossmann-Royer} as the weak interpretation of the Grossmann-Royer operator. Therefore $R_g f$
is  the cross-Wigner distribution in disguise (Definition \ref{transformacije}
and \cite[Definition 12]{deGosson2017}), so all the results valid for  $ W(f,g)$ can be reformulated in terms of  $R_g f$.
In fact,  $R_g f$ is also closely related to other time-frequency representations,
$V_g f$ and $A(f,g),$ see Lemma  \ref{GRandRelatives}.
We  perform a  detailed exposition of the properties of
the Grossmann-Royer transform, such as marginal densities and the weak uncertainty principle, Lemma \ref{marginal} and Proposition \ref{uncertainty}, respectively.
The relation between the Grossmann-Royer transform and the Heisenberg-Weyl (displacement) operators via the symplectic Fourier transform is revealed in Proposition \ref{symplectic}.

We start Section \ref{Gelfand} with a motivation for the study of Gelfand-Shilov spaces, and proceed with a
review of their main properties, including their different characterizations (Theorem \ref{GS-characterization}),
the kernel theorem (Theorem \ref{kernelteorema}), and a novel (albeit expected)
characterization in terms of the Grossmann-Royer transform (Theorem \ref{nec-suf-cond}).

 In Section \ref{Modulation Spaces} we first discuss weight functions, and then define modulation spaces as a subset of  $ (\cS^{1/2}_{1/2})' (\rd)$, the dual space of the Gelfand-Shilov space $  \cS^{1/2}_{1/2} (\rd)$, see Definition \ref{defmodnorm}. This section is an exposition of known results relevant for the rest of the paper.
 In particular, we recall that Gelfand-Shilov spaces can be described as the projective and the inductive limits of modulation spaces, Theorem \ref{modproerties} and Remark \ref{GSandmod}. We also recall the (sharp) convolution estimates for modulation spaces (Proposition \ref{mconvmp} and Theorem \ref{mainconvolution}) which will be used in the study of localization operators in Section \ref{localization}.

In Section \ref{localization} we rewrite the definition of localization operators and Weyl pseudodifferential operators in terms of the Grossmann-Royer transform (Definition \ref{locopdef} and Lemma \ref{locopsame}), and recover the main relations between them (Lemmas \ref{Weyl psido char} and \ref{Weyl connection lemma}). Thereafter we
give an exposition of known results concerning the continuity of localization operators on modulation spaces
(Subsection \ref{cont-prop}),
and their compactness properties when the weights are chosen to be polynomial or exponential (Subsection \ref{schvonneumann}).
As a rule, we only sketch the proofs, and use the Grossmann-Royer transform formalism whenever convenient.

Finally, in Section \ref{extensions} we point out some directions that might be relevant for future investigations. This includes product formulas,  multilinear localization operators, continuity over quasianalytic classes, and extension to quasi-Banach modulation spaces.

\subsection{Notation} We define
$xy=x\cdot y$,  the scalar product on $\Ren$. Given a vector
$x=(x_1,\dots,x_d)\in\rd$, the partial derivative with respect to
$x_j$ is denoted by  $\partial _j = \frac{\partial}{\partial x_j}$.
 Given a multi-index
$p=(p_1,\dots,p_d)\geq 0$, i.e., $p \in\bN^d_0$ and $p_j \geq 0$, we write
$\partial^p=\partial^{p_1}_1\cdots\partial^{p_d}_d$ and  $x^p =(x_1,\dots,x_d)^{(p_1,\dots,p_d)}=\prod_{i=1}^d x_i^{p_i}$.
We write $h\cdot |x|^{1/\a}=\sum_{i=1}^d h_i |x_i|^{1/\a_i}$. Moreover, for $p \in\bN^d_0$ and $\a \in\R^d_+$, we set $(p!)^\a=(p_1!)^{\a_1}\dots (p_d!)^{\a_d}$, while as standard $p!=p_1!\dots p_d!$.
In the sequel, a real number $r\in\R_+$ may play the role of the vector
with constant components $r_j=r$, so for $\a\in\R^d_+$, by writing $\a>r$ we mean $\a_j>r$ for all $j=1,\dots,d$.

The Fourier transform is normalized to be ${\hat   {f}}(\o)=\Fur f(\o)=\int f(t)e^{-2\pi i t\o}dt$.
We use the brackets $\la f,g\ra$ to denote the extension of the inner product $\la f,g\ra=\int f(t){\overline
{g(t)}}dt$ on $L^2(\Ren)$ to the dual pairing between a test function space $ \mathcal A $ and its dual $ {\mathcal A}' $:
$ \langle \cdot, \cdot \rangle  = $ $ _{{\mathcal A}'}\langle \cdot, \overline{\cdot} \rangle _{\mathcal A}.$

By $ \check{f} $ we denote the reflection $\check{f} (x) = f (-x),$
and $\langle \cdot \rangle ^s$  are polynomial weights
$$
 \la \phas \ra^s= \la z\ra^s=(1+x^2+\o^2)^{s/2},\quad
   z=(x,\o)\in\Renn, \,\quad s\in\bR ,\,
 $$
 and $\langle x \rangle = \langle 1 + |x|^2\rangle ^{1/2}$, when $ x
\in \mathbb{R}^d.$

Translation and modulation operators, $T$ and $M$ are defined by
\begin{equation} \label{trans-mod}
T_x f(\cdot) = f(\cdot - x) \;\;\; \mbox{ and } \;\;\;
 M_x f(\cdot) = e^{2\pi i x \cdot} f(\cdot), \;\;\; x \in \mathbb{R}^d.
$$
The following relations hold
$$
M_y T_x  = e^{2\pi i x \cdot y } T_x M_y, \;\;
 (T_x f)\hat{} = M_{-x} \hat f, \;\;
  (M_x f)\hat{} = T_{x} \hat f, \;\;\;
  x,y \in \mathbb{R}^d,  f,g \in L^2 (\mathbb{R}^d).
\end{equation}

The singular values
$\{s_k(L)\}_{k=1}^\infty$ of a compact operator $L\in B(L^2(\Ren))$ are the
eigenvalues of the positive self-adjoint operator $\sqrt{L^*L}.$ For $1\leq
p<\infty$, the Schatten class $S_p$ is the space of all compact operators
whose singular values lie in $l^p$. For  consistency,  we define
$S_\infty : =B(L^2(\Ren))$ to be the space of bounded operators on
$\lrd $. In particular, $S_2$ is the space of Hilbert-Schmidt
operators, and
$S_1$ is the space of trace class operators.

Throughout the paper, we shall use the notation $A\lesssim B$ to
indicate $A\leq c B$ for a suitable constant $c>0$, whereas $A
\asymp B$ means that $c\inv A \leq B \leq c A$ for some $c\geq 1$. The
symbol $B_1 \hookrightarrow B_2$ denotes the continuous and dense embedding of
the topological vector space $B_1$ into $B_2$.

\subsection{Basic spaces} \label{basicspaces}

In general a weight $ w(\cdot) $ on $ \mathbb{R}^{d}$ is a non-negative and continuous function.
By $ L^p _w (\mathbb{R}^d)$, $ p \in [1,\infty] $ we denote the weighted Lebesgue space defined by the norm
$$
\| f \|_{L^p _w } = \| f w \|_{L^p} = \left ( \int |f(x)|^p w(x) ^p dx \right )^{1/p},
$$
with the usual modification when $ p=\infty$.

Similarly, the weighted  mixed-norm space $ L^{p,q} _w (\mathbb{R}^{2d})$, $ p,q \in [1,\infty] $
consists of (Lebesgue) measurable functions on $ \mathbb{R}^{2d}$ such that
$$
\| F \|_{ L^{p,q} _w} = \left ( \int_{ \mathbb{R}^{d}} \left ( \int_{ \mathbb{R}^{d}}
| F(x,\omega)|^p w(x,\omega) ^p dx \right )^{q/p} d\omega \right )^{1/q} < \infty.
$$
where $ w(x,\omega ) $ is a weight on $ \mathbb{R}^{2d}$.

In particular, when $ w(x,\omega) = \langle  x
\rangle ^t \langle \omega \rangle ^s, $ $s,t\in \mathbb{R},$ we use
the notation $ L^{p,q} _w (\mathbb{R}^{2d}) = L^{p,q} _{t,s}
(\mathbb{R}^{2d})$, and when $ w(x) = \langle  x \rangle ^t $, $t\in
\mathbb{R},$ we use the notation $ L^p _t (\mathbb{R}^d) $ instead.

The space of smooth functions with compact support on
$\rd$  is  denoted by $\cD(\rd)$. The Schwartz class is denoted by  $\sch(\Ren)$, the space of
tempered distributions by  $\sch'(\Ren)$.
Recall,  $\sch(\Ren)$ is a Fr\'echet space, the projective limit of spaces $\sch_p (\Ren)$,
$ p \in\bN _0,$ defined by the norms:
$$
\| \varphi \|_{\sch_p} = \sup_{|\alpha| \leq p} (1+ |x|^2)^{p/2} | \partial^\alpha \phi (x)| <\infty, \;\;\;
p\in\bN _0.
$$
Note that $\cD(\rd) \hookrightarrow \sch(\Ren)$.

The spaces ${\mathcal S} (\mathbb{R}^d) $ and  $ {\mathcal S}' (\mathbb{R}^d) $
play an important role in various applications since
the Fourier transform is a topological isomorphism between $
{\mathcal S} (\mathbb{R}^d) $ and $ {\mathcal S} (\mathbb{R}^d) $ which extends to a continuous linear transform
from $ {\mathcal S}' (\mathbb{R}^d) $ onto itself.

In order to deal with particular problems in applications
different generalizations of the Schwartz type spaces were proposed.
An example is given by the Gevrey classes given below. Gelfand-Shilov spaces are another important example, see  Section
\ref{Gelfand}.

By $\Omega$ we denote an open set in $ \bR^d,$ and $ K \Subset \Omega $ means that $K$ is  compact subset in
$ \Omega.$ For $ 1 < s < \infty $ we define the Gevrey class $ G^s (\Omega) $ by
$$
G^s (\Omega) = \{ \phi \in C^\infty(\Omega) \;\; | \;\; (\forall K \Subset \Omega ) (\exists C > 0)
( \exists h > 0 ) \;\;\;
\sup_{x \in K} \left | \partial^\alpha  \phi (x) \right| \leq C h^{|\alpha|} |\alpha| ! ^s
\}.
$$
We denote by $ G_0 ^s (\Omega) $  a subspace of $ G^s (\Omega) $ which consists of compactly supported functions.
We have
$ \displaystyle {\mathcal A} (\Omega) \hookrightarrow \cap_{s > 1} G^s (\Omega) $ and
$ \displaystyle \cup_{s \geq 1} G^s (\Omega) \hookrightarrow C^\infty (\Omega),$
where $ {\mathcal A} (\Omega)  $ denotes the space of analytic functions defined by
$$
{\mathcal A} (\Omega) = \{  \phi \in C^\infty(\Omega) \;\; | \;\; (\forall K \Subset \Omega ) (\exists C > 0)
( \exists h > 0 ) \;\;\;
\sup_{x \in K} \left | \partial^\alpha   \phi(x) \right| \leq C h^{|\alpha|} |\alpha| ! \}.
$$

\par

We end this section with test function spaces for the spaces of ultradistributions due to Komatsu \cite{K}.
In fact, both the Gevrey classes and the Gelfand-Shilov spaces can be viewed as particular cases of Komatsu's construction.
Let there be given an open set
$\Omega \subset \bR^d,$
and a sequence $ (N_q)_{q \in \mN_0} $
which satisfies (M.1) and (M.2), see Section \ref{Gelfand}.
The function $ \phi \in C^{\infty} (\Omega) $ is called
{\em ultradifferentiable} function of Beurling class $ (N_q) $ (respectively of Roumieu class  $ \{ N_q \} $)
if, for any $ K \subset\subset\Omega $
and for any $h > 0$  (respectively for some $ h >0 $),
$$
\| \phi \|_{N_q, K, h} = \sup_{x \in K, \alpha \in \Nn _0} \frac{|\partial^\alpha \phi (x)|}{h^{|\alpha|} N_{|\alpha|}}
< \infty.
$$
We say that $ \phi \in  \cE^{N_q, K, h } (\Om) $ if $ \| \phi \|_{N_q, K, h} < \i $ for given $K$ and $h>0$,
and define the following spaces of ultradifferentiable test functions:
$$
\cE ^{(N_q)} (\Om)
:= {\rm proj} \lim_{ K \subset\subset\Om} {\rm proj} \lim_{h \rightarrow 0} \cE^{N_q, K, h } (\Om);
$$
$$
\cE ^{\{N_q\}} (\Om) : = {\rm proj} \lim_{ K \subset\subset\Om}  {\rm ind} \lim_{h \rightarrow \infty}
\cE^{N_q, K, h } (\Om).
$$

\section{The Grossmann- Royer transform} \label{Grossmann-Royer}

In this section we consider  the Grossmann-Royer operator $ R f  (x,\omega) $ and introduce the corresponding transform. We refer to \cite{deGosson2011} for the basic properties of $ R f (x,\omega) $ and its relation to the Heiseberg-Weyl operator. In fact,
the Grossmann-Royer operator can be viewed as a kind of reflection (therefore we use the letter $R$ to denote it),
more precisely the conjugate of a reflection operator by a Heisengerg-Weyl operator, and the product formula (\cite[Proposition 150]{deGosson2011}) reveals that the product of two reflections is a translation. Moreover, it satisfies the symplectic covariance property, \cite[Proposition 150]{deGosson2011}.

Those physically plausible interpretations motivate us to define the corresponding transform, which is just the cross-Wigner distribution in disguise. As we shall see, in many situations it is more convenient to choose the the Grossmann-Royer transform
instead of some of its relatives, the cross-Wigner distribution, the cross-ambiguity function and the short-time Fourier transform.

The  Grossmann-Royer operator $ R : L^2 (\mathbb{R}^d) \rightarrow L^2 (\mathbb{R}^{2d})  $ is given by
\be \label{GROp}
R f (x,\omega) =  R  (f(t) )(x,\omega)  = e^{4\pi i\omega (t-x)} f(2x - t), \;\;\; f \in L^2 (\mathbb{R}^d),\; x,\omega \in \mathbb{R}^d.
\ee

\begin{definition} \label{transformacije}
Let there be given $f, g \in L^2 (\mathbb{R}^d).$ The
Grossmann-Royer transform of $f$ and $g$ is given by
\be \label{GRT}
R_g f (x,\omega) = R (f,g) (x,\omega)  = \langle R f, g \rangle   = \int e^{4\pi i  \omega(t-x)} f(2x- t) \overline{ g(t)} dt, \;\;\; x,\omega \in \mathbb{R}^d.
\ee
The short-time Fourier transform of
$ f  $ with respect to the window $g$  is given by
\be \label{GT}
V_g f (x,\omega) = \int e^{-2\pi i t \omega} f(t)  \overline{ g(t-x)}dt, \;\;\; x,\omega \in \mathbb{R}^d.
\ee
The cross--Wigner distribution of  $f$ and $g$ is
\be \label{WD}
W(f,g) (x,\omega)  = \int e^{-2\pi i \omega t} f(x+ \frac{t}{2})
\overline{ g(x- \frac{t}{2})} dt, \;\;\; x,\omega \in \mathbb{R}^d,
\ee
and the cross--ambiguity function of  $f$ and $g$ is
\be \label{FWT}
A (f,g) (x, \omega) = \int e^{-2\pi i \omega t}
f(t+ \frac{x}{2}) \overline{ g(t- \frac{x}{2})} dt, \;\;\; x,\omega \in \mathbb{R}^d.
\ee
\end{definition}

Note that $ R  (f(t) )(x,\omega)  = e^{-4\pi i\omega x}  M_{2\omega}(T_{2x} f)^{\check{}} (t) $ so that
\be \label{GRT-trans-mod}
R_g f (x,\omega)  =  e^{-4\pi i\omega x} \langle M_{2\omega} (T_{2x} f)^{\check{}} , \overline{g} \rangle.
\ee

\begin{lemma} \label{GRandRelatives}
Let  $f, g \in L^2 (\mathbb{R}^d).$ Then we have:
\begin{eqnarray*}
W(f,g) (x,\omega)  & = & 2^d R_g f (x,\omega), \\
V_g f (x,\omega) & = &  e^{-\pi i x \omega} R_{\check{g}} f (\frac{x}{2},\frac{\omega}{2}),
\\
A(f,g) (x,\omega) & = & R_{\check{g}} f (\frac{x}{2},\frac{\omega}{2}), \;\;\;
x,\omega \in \mathbb{R}^d.
\end{eqnarray*}
\end{lemma}

\begin{proof}
$
W(f,g) (x,\omega)  = 2^d R_g f (x,\omega)$
immediately follows from the change of variables $t \mapsto 2(x-t)$ in \eqref{WD}, so  the
cross--Wigner distribution and the  Grossmann-Royer transform are essentially the same.

The change of variables $x-t \mapsto s$ in \eqref{GT} gives
\begin{eqnarray*}
V_g f (x,\omega) & = &  \int e^{-2\pi i (x-s) \omega}   f(x-s) \overline{\check{g}(s)} ds \\
 & = &  \int e^{4\pi i \frac{\omega}{2} (s-x) }   f(2 \frac{x}{2}-s) \overline{\check{g}(s)} ds
\\
 & = &  e^{-\pi i x \omega} \int e^{4\pi i \frac{\omega}{2} (s-\frac{x}{2}) }   f(2 \frac{x}{2}-s) \overline{\check{g}(s)} ds
\\
 & = &  e^{-\pi i x \omega} R_{\check{g}} f (\frac{x}{2},\frac{\omega}{2}).
\end{eqnarray*}
Therefore, $  R_g f (x,\omega) =   e^{\pi i x \omega} V_{\check{g}} f (2x,2\omega).$
Finally, the change of variables $t - x/2 \mapsto - s$ in \eqref{FWT} gives
\begin{eqnarray*}
A(f,g) (x,\omega)   & = &   \int e^{2\pi i \omega (s- \frac{x}{2})}
f(x-s) \overline{ \check{g}(s)} ds\\
  & = &   \int e^{4\pi i \frac{\omega}{2} (s- \frac{x}{2})}
f(2 \frac{x}{2} - s) \overline{ \check{g}(s)} ds\\
 & = &
 R_{\check{g}} f (\frac{x}{2},\frac{\omega}{2}).
\end{eqnarray*}
Therefore $  R_g f (x,\omega) =   A(f, \check{g}) (2x,2\omega).$
\qed
\end{proof}

By Lemma \ref{GRandRelatives} we recapture the well-known formulas (\cite{Gro, Wong1998}:
$$
A(f,g) (x,\omega)  = e^{\pi i x \omega} V_g f (x,\omega),
$$
$$
W(f,g) (x,\omega)  = 2^d e^{4\pi i x \omega} V_{\check{g}} f (2x,2\omega),
$$
$$
W(f,g) (x,\omega)  = ({\mathcal F} A(f,g) )(x,\omega), \;\;\;
x,\omega \in \mathbb{R}^d.
$$
The quadratic expressions $ A f := A (f,f) $ and  $ W f := W (f,f) $
are called the (radar) ambiguity function and the Wigner distribution of $f.$

We collect the elementary properties of the   Grossmann-Royer transform in the next proposition.

\begin{proposition} \label{svojstva}
Let  $f, g \in L^2 (\mathbb{R}^d).$ The Grossmann-Royer operator is self-adjoint, uniformly continuous on
$\mathbb{R}^{2d}$, and the following properties hold:
\begin{enumerate}
\item $ \| R_g f \|_\infty \leq \| f\| \|g\|$;
\item $  R_g f = \overline{R_f g}$;
\item  The covariance property:
$$
R_{T_x M_\omega g} T_x M_\omega f (p,q) = R_g f (p-x, q-\omega), \;\;\;\; x,\omega,p,q \in \mathbb{R}^d;
$$
\item $  R_{ \hat g} \hat f ( x, \omega) =  R_{ g} f (-\omega, x)$;
\item For $ f_1, f_2, g_1, g_2  \in L^2 (\mathbb{R}^d),$ the Moyal identity holds:
$$ \langle R_{g_1} f_1, R_{ g_2} f_2 \rangle = \langle f_1, f_2 \rangle \overline{\langle g_1, g_2 \rangle}. $$
\item $ R_g f$  maps $ {\mathcal S} (\mathbb{R}^d) \times
{\mathcal S} (\mathbb{R}^d) $ into $  {\mathcal S} (\mathbb{R} ^d \times \mathbb{R} ^d)$ and extends to
a map from $ {\mathcal S}' (\mathbb{R}^d) \times {\mathcal S}' (\mathbb{R}^d) $ into
$ {\mathcal S}' (\mathbb{R} ^d \times \mathbb{R} ^d).$
\end{enumerate}
\end{proposition}

\begin{proof}
We first show that $  \langle R f, g \rangle   =  \langle  f, R g \rangle:$
\begin{eqnarray*}
\langle R f, g \rangle
& = & \int e^{4\pi i  \omega(t-x)} f(2x- t) \overline{ g(t)} dt =
\int e^{4\pi i  \omega(x-s)} f(s) \overline{ g(2x-s)} ds
\\
& = & \int  f(s)  \overline{ e^{4\pi i  \omega(s-x)}  g(2x-s)} ds  = \langle  f, R g \rangle.
\end{eqnarray*}

The uniform continuity and the estimate in {\em 1.} follow from \eqref{GRT-trans-mod}, the continuity of  translation, modulation and reflection operators on $ L^2 (\mathbb{R}^d)$ and the Cauchy-Schwartz inequality.

The property {\em 2.} follows from the self-adjointness.

For the covariance we have:
\begin{eqnarray*}
R_{T_x M_\omega g} T_x M_\omega f (p,q)
& = & \int e^{4\pi i q(t-p)} T_x M_\omega f (2p -t) \overline{ T_x M_\omega g} (t) dt \\
& = & \int e^{4\pi i q(t-p)} e^{2\pi i \omega (2p-t)}  f (2p -t-x) e^{-2\pi i \omega t}  \overline{ g} (t-x) dt \\
& = & \int e^{4\pi i q(s+x-p)} e^{2\pi i \omega (2p-s-x)}  f (2(p-x)-s) e^{-2\pi i \omega (s+x)}  \overline{ g} (s) ds \\
& = & \int e^{4\pi i (q-\omega)(s-(p-x))}  f (2(p-x)-s)   \overline{ g} (s) ds \\
& = &  R_g f (p-x, q-\omega), \;\;\;\; x,\omega,p,q \in \mathbb{R}^d.
\end{eqnarray*}

To prove {\em 4.} we note that the integrals below are absolutely convergent and that the change of order of integration
is allowed. Moreover, when suitably interpreted, certain oscillatory integrals are meaningful in
$  {\mathcal S}^{'} (\mathbb{R}^{d})$. We refer to e.g. \cite{Shubin91, Wong1998} for such interpretation. In particular, if $\delta$ denotes the Dirac distribution, then
the Fourier inversion formula in the sense of distributions gives
$
\int e^{2\pi i x \omega} d \omega = \delta (x),
$
wherefrom $ \iint \phi (x)  e^{2\pi i (x-y) \omega} d x d \omega = \phi (y),$ when $\phi \in
{\mathcal S}(\mathbb{R}^{d})$.

Therefore we have:
\begin{eqnarray*}
R_{ \hat g} \hat f ( x, \omega)
& = & \int e^{4\pi i  \omega(t-x)} \hat{f}(2x- t) \overline{ \hat{g}(t)} dt \\
& = & e^{-4\pi i  \omega x} \int \int \int e^{4\pi i  \omega t }
e^{-2\pi i  \eta (2x-t) } f(\eta) e^{2\pi i  y t }  \overline{ g(y)} d\eta dy dt \\
& = & e^{-4\pi i  \omega x} \int \int \int
e^{-2\pi i t( -2 \omega - \eta -y) } e^{-4\pi i  \eta x} f(\eta) \overline{ g(y)} d\eta dy dt \\
& = & e^{-4\pi i  \omega x} \int \int
\delta ( -2 \omega - \eta -y)  e^{-4\pi i  \eta x} f(\eta) \overline{ g(y)} d\eta dy \\
& = & e^{-4\pi i  \omega x} \int
e^{-4\pi i  ( -2 \omega-y) x}  f( -2 \omega-y) \overline{ g(y)}  dy \\
& = &  \int
e^{4\pi i  (y+ \omega) x}  f( -2 \omega-y) \overline{ g(y)}  dy \\
& = &   R_{ g} f (-\omega, x), \;\;\;\; x,\omega \in \mathbb{R}^d.
\end{eqnarray*}

For the Moyal identity, we again use the Fourier transform of Dirac's delta:
\begin{eqnarray*}
\langle R_{g_1} f_1, R_{ g_2} f_2 \rangle
& = & \int_{\mathbb{R}^{2d}} \big (
\int e^{4\pi i  \omega(t-x)} f_1 (2x- t) \overline{ g_1(t)} dt  \\
& \times &
\int e^{-4\pi i  \omega(s-x)} \overline{f_2} (2x- s) g_2(s) ds  \big ) dx d\omega \\
& = &
\int_{\mathbb{R}^{2d}} \iint  e^{-2\pi i  \omega (-2(t-s))}
 f_1 (2x- t)  \overline{f_2} (2x- s) \overline{ g_1(t)} g_2(s) dt ds
  dx d\omega \\
& = &
\int \int \int \delta (-2(t-s)) f_1 (2x- t)  \overline{f_2} (2x- s) \overline{ g_1(t)} g_2(s) dt ds
 dx  \\
 & = &
 \langle f_1, f_2 \rangle \overline{\langle g_1, g_2 \rangle}.
\end{eqnarray*}

To prove {\em 6.} we note that the Grossman-Royer transform can be written in the form
\be \label{GRrewritten}
R_g f (x,\omega) = (\mathcal{F}_2 \circ \tau^*) f \otimes \overline{ g} (x,\omega), \;\;\;x,\omega \in \mathbb{R}^d,
\ee
where $ \mathcal{F}_2 $ denotes the partial Fourier transform with respect to the second variable, and
$ \tau^*$ is the pullback of the operator $ \tau : (x,s) \mapsto (2x-s,s).$ The theorem now follows
from the invariance of ${\mathcal S} (\mathbb{R}^{2d}) $ under  $ \tau^* $ and  $ \mathcal{F}_2 $,
and $f \otimes \overline{ g} \in  {\mathcal S} (\mathbb{R}^{2d}) $ when $f,g \in  {\mathcal S} (\mathbb{R}^d),$ and
the extension to tempered distributions is straightforward.
\qed
\end{proof}

The Moyal identity formula implies the inversion formula:
$$
f = \frac{1}{\langle g_2, g_1 \rangle} \int \int  R_{g_1} f (x, \omega)
R g_2 (x,\omega)  dx d\omega,
$$
We refer to \cite{Gro} for details (explained in terms of the short-time Fourier transform).

For  the Fourier transform of $R_g f $ we have:

\begin{proposition}
Let  $f, g \in L^2 (\mathbb{R}^d).$ Then $\Fur (R_g f) (x,\omega) = R_{\check{g}} f (- \frac{\omega}{2}, \frac{x}{2}), $  $ x,\omega \in \mathbb{R}^d.$
\end{proposition}

\begin{proof} We again use the  interpretation of oscillatory integrals in distributional sense.
\begin{eqnarray*}
\Fur (R_g f) (x,\omega)
& = & \int \int  e^{-2\pi i  (x x' + \omega \omega')} R_g f (x', \omega') dx' d\omega' \\
& = & \int \int \int   e^{-2\pi i  (x x' + \omega \omega') + 4\pi i \omega' (t-x')} f (2x' -t) \overline{g} (t) dt dx' d\omega' \\
& = & \int \int \int e^{-2\pi i  \omega' (\omega-2t+2x')} e^{-2\pi i x x'} f (2x' -t) \overline{g} (t) dt dx' d\omega' \\
& = & \int \int \delta (\omega-2t+2x') e^{-2\pi i x x'} f (2x' -t) \overline{g} (t) dt dx' \\
& = & \int \int  e^{-2\pi i x (t - \omega/2)} f (t-\omega) \overline{g} (t) dt  \\
& = & \int \int  e^{-4\pi i \frac{x}{2} (t - \omega/2)} f (t-\omega) \overline{g} (t) dt  \\
& = & \int \int  e^{4\pi i \frac{x}{2} (s -(- \omega/2))} f (2\cdot (- \omega/2) -s) \overline{\check{g}} (s) ds  \\
& = &  R_{\check{g}} f (- \frac{\omega}{2}, \frac{x}{2}), \;\;\;  x,\omega \in \mathbb{R}^d,
\end{eqnarray*}
which proves the proposition.
\qed
\end{proof}

The following lemma shows that the Grossmann-Royer transform yields the correct marginal densities.
This follows from the well-known  marginal densities of the (cross-)Wigner distribution,
since the transforms are essentially the same. However, we give an independent proof below.
We refer to \cite{Gro, deGosson2011}
for the discussion on the probabilistic interpretation of the (cross-)Wigner distribution.

\begin{lemma} \label{marginal}
Let $f,g \in L^1 \mathbb{R}^{d} \cap L^2 \mathbb{R}^{d}.$ Then
$$
\int_{\mathbb{R}^d} R_g f (x,\omega) d\omega = 2^{-d} f(x) \overline{g} (x),
$$
$$
\int_{\mathbb{R}^d} R_g f (x,\omega) dx = 2^{-d} \hat{f}(\omega) \overline{\hat{g}} (\omega),
$$
and, in particular $ 2^{d } \int R_f f (x,\omega) d\omega = |f(x)|^2, $ and
$
2^{d } \int R_f f (x,\omega) dx = |\hat{f}(\omega)|^2.
$
\end{lemma}

\begin{proof}
We will use the change of variables $s = x - t $ to obtain
\begin{eqnarray*}
\int_{\mathbb{R}^d} R_g f (x,\omega) d\omega
& = & \int \int    e^{ 4\pi i \omega (t-x)} f (2x -t) \overline{g} (t)  dt d\omega \\
& = & \int \int e^{-2\pi i  \omega (2s)} f (x+s) \overline{g} (x-s) ds  d\omega \\
& = & \int \delta (2s)  f (x+s) \overline{g} (x-s) ds\\
& = &  2^{-d} f (x) \overline{g} (x), \;\;\;  x \in \mathbb{R}^d,
\end{eqnarray*}
and similarly
\begin{eqnarray*}
\int_{\mathbb{R}^d} R_g f (x,\omega) dx
& = & \int \int    e^{ -2\pi i \omega (2x - 2t)} f (2x -t) \overline{g} (t)  dt dx \\
& = & \int \int    e^{ -2\pi i \omega (2x - t)} f (2x -t) \overline{ e^{ -2\pi i \omega t} g(t)}  dt dx \\
& = & 2^{-d} \hat{f}(\omega) \overline{\hat{g}} (\omega), \;\;\;  \omega \in \mathbb{R}^d.
\end{eqnarray*}
The particular case is obvious.
\qed
\end{proof}

\begin{remark} If $f$ belongs to the dense subspace of $ L^2 ( \mathbb{R}^{d})$ such that $ R_f f \in L^1 (\mathbb{R}^{2d})$ and the Fubini theorem holds, then  Plancherel's theorem follows from Lemma \ref{marginal}:
$$
\| \hat f\|^2 = 2^{d} \int \int R_f f (x,\omega) dx d\omega = 2^{d} \int \int R_f f (x,\omega)  d\omega dx
= \|  f\|^2.
$$
\end{remark}

We end the section with the weak form of the uncertainty principle for the Grossmann-Royer transform.

\begin{proposition} \label{uncertainty}
Let there be given  $f, g \in L^2 (\mathbb{R}^d)\setminus 0,$ and let $ U\subset \mathbb{R}^{2d} $  and $\varepsilon > 0 $
be such that
$$
\int_{U} | R_g f (x, \omega) |^2 dx d\omega \geq (1 - \varepsilon) \| f\|  \|g\|.
$$
Then $|U| \geq 1- \varepsilon.$
\end{proposition}

\begin{proof}
From Proposition \ref{svojstva} {\em 1.} we have
$$
(1 - \varepsilon) \| f\|  \|g\| \leq \int_{U} | R_g f (x, \omega) |^2 dx d\omega \leq
\| R_g f \|_\infty  |U| \leq  \| f\| \|g\| |U|,
$$
and the claim follows.
\qed
\end{proof}

We end this section with the relation between the Grossmann-Royer operator and the Heisenberg-Weyl  operator, also known as displacement operators since it describes translations in phase space. Notice that  in \cite{deGosson2017}
the Grossmann-Royer operators are defined in terms of the Heisenberg-Weyl  operators,
so our definition \eqref{GROp} is formulated as a proposition  there.

\begin{definition} Let there be given $f, g \in L^2 (\mathbb{R}^d).$
The Heisenberg-Weyl  operator $ T : L^2 (\mathbb{R}^d) \rightarrow L^2 (\mathbb{R}^{2d})  $ is given by
\be \label{HWOp}
T f (x,\omega) =  T  (f(t) )(x,\omega)  = e^{2\pi i\omega (t-\frac{x}{2})} f(t-x), \;\;\;  x,\omega \in \mathbb{R}^d,
\ee
and the Heisenberg-Weyl transform is defined to be
\be \label{HWT}
T_g f (x,\omega) = T (f,g) (x,\omega)  = \langle T f, g \rangle   = \int e^{2\pi i\omega (t-\frac{x}{2})} f(t-x) \overline{ g(t)} dt, \;\;\; x,\omega \in \mathbb{R}^d.
\ee
\end{definition}

\begin{proposition}
Let there be given  $f, g \in L^2 (\mathbb{R}^d).$ Then we have
\begin{enumerate}
\item $ R f (x,\omega) = T (x,\omega)  R(0) T  f (-x,-\omega)$;
\item $ R (x,\omega) R f (p,q) = e^{-4\pi i \sigma ((x,\omega), (p,q))}  T  f (2(x-p,\omega-q)), $
where $\sigma$ is the standard symplectic form on phase space $ \mathbb{R}^{2d}$:
$$
\sigma ((x,\omega), (p,q)) = \omega \cdot  p - x \cdot q.
$$
\end{enumerate}
\end{proposition}

\begin{proof} {\em 1.}
Note that $R f(0) = f(-t)$, wherefrom
$$
T  f (-x,-\omega) =  e^{-2\pi i\omega (t+\frac{x}{2})} f(t+x) \;\;\; \Rightarrow
 R(0) T  f (-x,-\omega) =  e^{2\pi i\omega (t -\frac{x}{2})} f(x-t).
$$
Hence,
\begin{eqnarray*}
T (x,\omega)  R(0) T  f (-x,-\omega) & = & T (x,\omega)  e^{2\pi i\omega (t -\frac{x}{2})} f(x-t) \\
& = & e^{-\pi i\omega x}  T ( e^{2\pi i\omega t} f(x-t) ) (x,\omega) \\
& = & e^{-\pi i\omega x} e^{2\pi i\omega (t-\frac{x}{2})}   e^{2\pi i\omega (t-x)} f(x-(t-x))   \\
& = &  e^{4\pi i\omega (t-x)} f(2x-t)   \\
& = &   R f (x,\omega).
\end{eqnarray*}
To prove {\em 2.} we calculate both sides:
\begin{eqnarray*}
R (x,\omega) R f (p,q)  & = & R (e^{4\pi iq (t-p)} f(2p-t)) (x,\omega) \\
& = & e^{-4\pi iqp} R (e^{4\pi iq t} f(2p-t)) (x,\omega) \\
& = & e^{-4\pi iqp} e^{4\pi i\omega (t-x)}   e^{4\pi iq(2x-t)} f(2p-(2x-t)) \\
& = & e^{4\pi i\omega (t-x)}   e^{4\pi iq(2x-t-p)} f(t -2(x-p)),
\end{eqnarray*}
and
\begin{multline*}
e^{-4\pi i (\omega  p - x q)}  T  f (2((x,\omega) - (p,q)))  =
e^{-4\pi i (\omega  p - x q)}  e^{2\pi i 2(\omega - q) (t- (x-p))} f(t-2(x-p))\\
 =  e^{-4\pi i (\omega  p - x q)}  e^{4\pi i (\omega - q) t} e^{-4\pi i (\omega - q)(x-p)} f(t-2(x-p))\\
 =  e^{4\pi i \omega(- p +t- x-p)}   e^{-4\pi i p q}  e^{4\pi i q(x-t +x)}  f(t-2(x-p))\\
 =  e^{4\pi i\omega (t-x)}   e^{4\pi iq(2x-t-p)}   f(t-2(x-p)),
\end{multline*}
which shows that
$ R (x,\omega) R f (p,q) = e^{-4\pi i \sigma ((x,\omega), (p,q))}  T  f (2(x-p,\omega-q)). $
\qed
\end{proof}

Finally, if $\Fur _\sigma $ denotes the symplectic Fourier transform in $ L^2 (\mR^{2d}) $:
\begin{eqnarray*}
\Fur _\sigma ( F(x,\omega)) (p,q) & = & \int_{\mR^{2d}}
e^{-2\pi i \sigma ((x,\omega), (p,q))} F(x,\omega)
dx d\omega \\
& = & \int_{\mR^{2d}} e^{-2\pi i (\omega  p - x  q)} F(x,\omega)
dx d\omega,
\end{eqnarray*}
then the Grossmann-Royer operator and the Heisenberg-Weyl  operator are the symplectic Fourier transforms of each other.

\begin{proposition} \label{symplectic}
Let there be given  $f, g \in L^2 (\mathbb{R}^d)$ and let $\Fur _\sigma $ be the symplectic Fourier transform. Then
$$
R f (p,q) = 2^{-d} \Fur _\sigma  (T  (f(t) )(x,\omega) ) (-p,-q).
$$
\end{proposition}
\begin{proof}
\begin{eqnarray*}
\Fur _\sigma  (T  (f(t) )(x,\omega) ) (-p,-q) & = &  \int_{\mR^{2d}} e^{2\pi i (x  q -\omega  p  )}
e^{2\pi i \omega (t- \frac{x}{2})} f(t-x) dx d\omega
\\
& = &  \int_{\mR^{d}} \left ( \int_{\mR^{d}} e^{2\pi i \omega(t-  p -\frac{x}{2} )}  d\omega \right )
e^{2\pi i q x)} f(t-x) dx \\
& = &  \int_{\mR^{d}} \delta (t-  p -\frac{x}{2} ) e^{2\pi i q x} f(t-x) dx \\
& = &  2^d e^{4\pi i q (t-p)} f(t-2(t-p)) dx = 2^d  R f (p,q),
\end{eqnarray*}
where we again used the Fourier inversion formula in the sence of distributions.
\qed
\end{proof}

Since it is easy to see that $\Fur _\sigma $ is an involution, we also have
$$
T f (p,q) = 2^{d} \Fur _\sigma  (R  (f(t) )(x,\omega) ) (-p,-q).
$$

\begin{remark}
The Weyl operator $L_a $ with the symbol $a$ is introduced in \cite[Definition 37]{deGosson2017} by the means of the
symplectic Fourier transform $\Fur _\sigma a$ of the symbol and the Heisenberg-Weyl  operator:
$$ L_a f (t) = \int_{\mR^{2d}} \Fur _\sigma a (x, \omega) T (f(t) )(x,\omega) dx d\omega.$$
There it is shown that such definition coincides with the usual one, see Section \ref{localization},
as well as with the representation given by Lemma \ref{Weyl psido char}.
\end{remark}

\section{Gelfand--Shilov spaces} \label{Gelfand}

Problems of regularity of solutions to partial differential equations (PDEs)
play a central role in the modern theory of PDEs. When solutions of certain PDEs are smooth but not analytic,
several intermediate spaces of functions are proposed in order to describe its decrease at infinity
and also the regularity in $ \mR ^d$. In particular,
in the study of properties of solutions of certain parabolic initial-value
problems Gelfand and Shilov introduced
{\em the spaces of type $S$}
in \cite{GS55}.  Such spaces provide uniqueness and correctness classes for
solutions of Cauchy problems, \cite{GS}.
We refer to  \cite{GS} for the fundamental properties of such spaces which are
afterwards called Gelfand-Shilov spaces.

More recently, Gelfand-Shilov spaces were used in \cite{CGR-1,CGR-2}
to describe exponential decay and holomorphic extension of solutions to globally
elliptic equations, and in \cite{LMPSX} in the regularizing properties of the Boltzmann equation.
We refer to \cite{NR} for a recent overview  and for applications in quantum mechanics and traveling waves,
and to \cite{Toft-2012} for the properties of the Bargmann transform on Gelfand-Shilov spaces.
The original definition is generalized already in \cite[Ch. IV, App. 1]{GS},
and more general decay and regularity conditions can be systematically  studied
by using Komatsu's approach to ultradifferentiable functions developed in \cite{K}.
An interesting extension based on the iterates of the harmonic oscillator is studied in \cite{Toft-2017}
under the name Pilipovi\'c spaces, see also \cite{P1}.

\par

In the context of time-frequency analysis, certian Gelfand-Shilov spaces
can be described in terms of modulation spaces \cite{Gro,GZ}.
The corresponding pseudodifferential calculus is developed in \cite{Toft-2012,Toft-2}.
Gelfand-Shilov spaces are also used in the study of
time-frequency localization operators in \cite{CPRT2, CPRT1,Teof2015}, thus extending the
context of the pioneering results of Cordero and Grochenig \cite{CG02}.

In this section we introduce Gelfand-Shilov spaces
and list important equivalent characterizations. We also present the kernel theorem
which will be used in the study of localization operators.

\subsection{Definition}

The regularity and decay properties of elements of Gelfand-Shilov spaces
are initially measured with respect to sequences of the form $ M_p = p^{\alpha p},$
$ p \in \mN,$ $ \alpha > 0 $ or, equivalently, the Gevrey sequences
$ M_p = p!^{\alpha},$ $ p \in \mN,$ $ \alpha > 0 $.

We follow here Komatsu's approach \cite{K} to spaces of ultra-differentiable functions
to extend the original definition as follows.

Let $ (M_p)_{p \in \mN_0} $ be a sequence of positive numbers monotonically
increasing to infinity which satisfies:

\noindent $(M.1)\;\;$ $ M_p ^2 \leq M_{p-1} M_{p+1}, \;\;\; p \in
\mN; $

\noindent $(M.2) \;\;$ There exist positive constants $ A,H $ such
that
$$
 M_{p} \leq A H^p \mbox{ min }_{0 \leq q \leq p} M_{p-q} M_{q},
\;\; p,q \in \mN_0,
$$
or, equivalently, there exist positive constants $ A,H $ such that
$$
 M_{p+q} \leq A H^{p+q} M_{p} M_{q},
\;\; p,q \in \mN_0;
$$
We assume $ M_0 = 1, $ and that $ M_p ^{1/p } $ is bounded below by
a positive constant.

\begin{remark}
To give an example, we describe (M.1) and (M.2) as follows. Let
$(s_p)_{p \in \mN_0}$ be a sequence of positive numbers monotonically increasing to infinity
($ s_p \nearrow \infty $) so that for every $p, q
\in \mathbb{N}_0$ there exist $A, H>0$ such that
\begin{equation} \label{(M.2)}
\prod_{j=1} ^{q} s_{p+j} =  s_{p+1} \cdots s_{p+q}\leq AH^{p} s_{1} \cdots s_{q} =  AH^{p} \prod_{j=1} ^{q} s_{j}.
\end{equation}
Then  the sequence $ (S_p)_{p \in \mN_0} $ given by $S_p =  \prod_{j=1} ^{p} s_j$, $S_0=1,$
satisfies  $(M.1)$ and $(M.2) $.

Conversely, if $ (S_p)_{p \in \mN_0} $ given by $S_p =  \prod_{j=1} ^{p} s_j$, $s_j > 0, $ $ j \in \mathbb{N},$ $S_0=1,$
satisfies $(M.1)$ then $ (s_p)_{p \in \mN_0} $ increases to
infinity, and if it satisfies $(M.2)$ then (\ref{(M.2)}) holds.
\end{remark}

Let $ (M_p)_{p \in \mN_0} $ and $ (N_q)_{q \in \mN_0} $ be sequences
which satisfy $ (M.1). $ We write $ M_p \subset N_q $ ($ (M_p) \prec
(N_q), $ respectively) if there are constants $ H,C > 0 $ (for any
$H>0$ there is a constant $ C>0,$ respectively) such that $ M_p \leq
C H^p N_p, $ $ p \in \mN_0.$ Also, $ (M_p)_{p \in \mN_0} $ and $
(N_q)_{q \in \mN_0} $ are said to be equivalent if $ M_p \subset N_q
$ and $ N_q \subset M_p $ hold.

\begin{definition} \label{GSoftypeS}
Let there be given sequences of positive numbers $ (M_p)_{p \in
\mN_0} $ and $ (N_q)_{q \in \mN_0} $ which satisfy $ (M.1) $ and $
(M.2).$ Let $ \cS^{N_{q} , B} _{M_{p}, A}
(\R^d) $ be defined by
$$
 \cS^{N_{q} , B} _{M_{p}, A} (\R^d) =
\{ f \in C^\infty(\R^d) \; |\;
 \| x^{\alpha} \partial^{\beta}  f  \|_{L^\infty} \leq
 C A^{\alpha} M_{|\alpha|} B^{\beta} N_{|\beta|}, \;\;
\forall \alpha,\beta \in \N_{0} ^d \},
$$
for some positive constant $C,$ where $ A = (A_1,\dots,A_d),$ $ B =
(B_1,\dots,B_d),$ $A, B>0.$

Gelfand-Shilov spaces $ \Sigma^{N_q} _{M_p} (\R^d) $ and $  \cS^{N_q} _{M_p} (\R^d) $
are projective and inductive limits
of the (Fr\'echet) spaces $ \cS^{N_{q} , B} _{M_{p}, A}
(\R^d) $ with respect to $A$ and $B$:
$$
\Sigma^{N_q} _{M_p} (\R^d) : = {\rm proj} \lim_{A>0, B>0} \cS^{N_q , B}
_{M_p, A} (\R^d) ; \;\;\; \cS^{N_q} _{M_p} (\R^d)  : = {\rm ind} \lim_{A>0, B>0}
\cS^{N_q , B} _{M_p, A}  (\R^d).
$$

The corresponding dual spaces of $\Sigma^{N_q} _{M_p} (\R^d) $ and $\cS^{N_q} _{M_p} (\R^d) $
are the spaces of ultradistributions of Beurling and Roumier type respectively:
$$
(\Sigma^{N_q} _{M_p} )' (\R^d) : = {\rm ind} \lim_{A>0, B>0}
(\cS^{N_q , B} _{M_p, A})' (\R^d) ;
$$
$$
(\cS^{N_q} _{M_p})' (\R^d)  : = {\rm proj} \lim_{A>0, B>0}
(\cS^{N_q , B} _{M_p, A})'  (\R^d).
$$
\end{definition}

\par

Of course, for certain choices of the sequences $ (M_p)_{p \in \mN_0} $ and $ (N_q)_{q \in \mN_0} $
the spaces $ \Sigma^{N_q} _{M_p} (\R^d) $ and $  \cS^{N_q} _{M_p} (\R^d) $  are trivial, i.e. they
contain only the function $ \phi \equiv 0$. Nontrivial Gelfand-Shilov spaces are closed under translation, dilation,
multiplication with $x\in \R^d,$ and differentiation. Moreover, they are closed under the action of certain differential
operators of infinite order (ultradifferentiable operators in the terminology of Komatsu).
We refer to \cite{K} for topological properties in a more general context of test function spaces for ultradistributions.

When $ (M_p)_{p \in
\mN_0} $ and $ (N_q)_{q \in \mN_0} $ are Gevrey sequences:
$ M_p = p!^{r}, $ $ p\in \mN_{0}$  and $ N_q = q!^{s}, $  $ q\in
\mN_{0}$, for some $r,s\geq 0$, then we use the notation
$$ {\mathcal S}^{N_q} _{M_p} (\R^d) =
{\mathcal S}^{s} _{r} (\R^d) \;\;\; \text{and} \;\;\; \Sigma^{N_q} _{M_p} (\R^d) =  \Sigma^{s} _{r} (\R^d).
$$
If, in addition, $ s = r, $ then we put
$$ {\mathcal S}^{\{ s\}}  (\R^d) =
{\mathcal S}^{s} _{s} (\R^d)  \;\;\; \text{and} \;\;\; \Sigma^{(s)} (\R^d) =  \Sigma^{s} _{s} (\R^d).
$$

The choice of Gevrey sequences (which is the most often used choice in the literature)
may serve well as an illuminating example in different contexts. In particular, when discussing
the nontriviality  we have the following:
\begin{enumerate}
\item the space $ {\mathcal S}^{s} _{r} (\R^d)$
is nontrivial if and only if $ s+r > 1 $, or $ s+r = 1$ and $sr>0$,
\item if $ s+r \geq 1 $ and $s<1$, then every  $ f \in {\mathcal S}^{s} _{r} (\R^d)$
can be extended to the complex domain as an entire function,
\item if $ s+r \geq 1 $ and $s=1$, then every  $ f \in {\mathcal S}^{s} _{r} (\R^d)$
can be extended to the complex domain as a holomorphic function in a strip.
\item the space $ \Sigma^{s} _{r} (\R^d)$
is nontrivial if and only if $ s+r > 1 $, or, if $ s+r = 1$ and $sr>0$ and $ (s,r)\neq (1/2,1/2)$.
\end{enumerate}

We refer to \cite{GS} or \cite{NR} for the proof in the case of $ {\mathcal S}^{s} _{r} (\R^d)$,
and to \cite{P1} for the spaces $  \Sigma^{s} _{r} (\R^d) $, see also \cite{Toft-2012}.

Whenever nontrivial, Gelfand-Shilov spaces contain "enough functions" in the following sense.
A test function space $ \Phi $ is "rich enough" if
$$ \int f(x) \varphi (x) dx = 0, \;\;\; \forall   \varphi \in \Phi
\Rightarrow f(x) \equiv 0 (a.e.).
$$

The discussion here above shows that Gelfand-Shilov classes $ {\mathcal S}^{s} _{r} (\R^d)$
consist of quasi-analytic functions when  $s \in (0,1)$.
This is in sharp contrast with e.g. Gevrey classes
$  G^{s} (\mathbb{R}^d),$ $s>1$, another family of functions commonly used in regularity theory
of partial differential equations, whose elements are always non-quasi-analytic.
We refer to \cite{R} for microlocal analysis in Gervey classes and note that
$ G^{s} _0  (\mathbb{R}^d) \hookrightarrow {\mathcal S}_s ^s (\mathbb{R}^d)
\hookrightarrow  G^{s} (\mathbb{R}^d),$ $ s > 1.$

\par

When the spaces are nontrivial we have dense and continuous
inclusions:
$$
\Sigma^{s} _{r}  (\R^d) \hookrightarrow {\mathcal S}^{s} _{r} (\R^d) \hookrightarrow {\mathcal S}  (\R^d).
$$

\par

In fact, $ {\mathcal S}  (\R^d) $ can be revealed as a limiting case of $ S^{s} _{r} (\R^d)$, i.e.
$$
{\mathcal S}  (\R^d) = {\mathcal S}^{\infty} _{\infty} (\R^d) = \lim_{s, r \rightarrow \infty} {\mathcal S}^{s} _{r} (\R^d),
$$
when the passage to the limit  $ s, r \rightarrow \infty$ is interpreted correctly, see \cite{GS}, page 169.

We refer to \cite{Toft-2017} where it is shown how to overcome the minimality condition
($ \Sigma^{1/2} _{1/2} (\R^d) = 0$) by transferring the estimates for
$ \| x^{\alpha} \partial^{\beta}  f  \|_{L^\infty} $ into the  estimates of the form
$ \|  H^N f  \|_{L^\infty}  \lesssim h^N (N!)^{2s},$ for some (for every ) $ h>0$,
where $H = |x|^{2} -  \Delta $ is the harmonic oscillator.

\par

In what follows, the special role will be played by the Gelfand-Shilov space of analytic functions   $  {\mathcal S} ^{(1)} (\mathbb{R}^d ) := \Sigma^{1} _{1}  (\R^d)$. According to Theorem \ref{GS-characterization} here below, we have
$$ f \in  {\mathcal S}^{( 1)} (\mathbb{R}^d) \Longleftrightarrow
\sup_{x\in \mathbb{R}^d } |f(x)  e^{h\cdot |x|}| < \infty \;
\; \text{and} \;
\sup_{\omega \in \mathbb{R}^d } | \hat f (\omega)  e^{h\cdot |\omega|} | < \infty, \;\; \forall  h > 0.
$$
Any $ f \in  {\mathcal S}^{( 1)} (\mathbb{R}^d) $ can be extended to a
holomorphic function $f(x+iy)$ in the strip
$ \{ x+iy \in \mC ^d \; : \; |y| < T \} $ some $ T>0$, \cite{GS, NR}.
The dual space of  $  {\mathcal S} ^{(1)} (\mathbb{R}^d ) $ will be denoted by
$ {\mathcal S} ^{(1)'}   (\mathbb{R}^d ). $
In fact,  $  {\mathcal S} ^{(1)} (\mathbb{R}^d ) $ is isomorphic to the Sato test function space for the space of Fourier hyperfunctions $ {\mathcal S} ^{(1)'}   (\mathbb{R}^d ), $ see \cite{CCK1994}.

\subsection{Equivalent conditions}

In this subsection we recall the well known equivalent characterization of Gelfand-Shilov spaces which
 shows the important behavior of Gelfand-Shilov spaces under the action of the Fourier transform.
Already in \cite{GS} it is shown that the Fourier transform is a topological isomorphism between
$ {\mathcal S}_r ^s  (\R^d)$ and $ {\mathcal S}_s ^r  (\R^d)$
$( {\mathcal F} ( {\mathcal S}_r ^s) = {\mathcal S}_s ^r)$,
which extends to a continuous linear transform
from $ ({\mathcal S}_r ^s) '  (\R^d)$ onto $ ({\mathcal S}_s ^r)'  (\R^d)$.
In particular, if $ s = r $ and $s \geq 1/2 $
then $ {\mathcal F} ( {\mathcal S}_s ^s) (\R^d)= {\mathcal S}_s ^s (\R^d).$
Similar assertions hold for $\Sigma^{s} _{r} (\R^d).$

This invariance properties easily follow from the following theorem which also
enlightens fundamental  properties of Gelfand-Shilov spaces implicitly contained in their definition.
Among other things, it states that the decay and regularity estimates of $f \in
{\mathcal S}^{N_q} _{M_p} (\R^d)$ can be studied separately.

Before we state the theorem, we introduce another notion.
The {\em associated function} for a given sequence  $ (M_p) $ is defined by
$$
M(\rho) = \sup_{p\in \mN_0} \ln \frac{\rho^p M_0}{M_p}, \;\;\; 0 <
\rho < \infty.
$$

For example, the associated function for the Gevrey sequence  $ M_p = p!^{r}, $ $ p\in \mN_{0}$ behaves
at infinity  as $ |\cdot|^{1/r},$ cf. \cite{Pe}. In fact, the interplay between the defining sequence and its associated function
plays an important role in the theory of ultradistributions.

\begin{theorem} \label{GS-characterization}
Let there be given sequences of positive numbers $ (M_p)_{p \in
\mN_0} $ and $ (N_q)_{q \in \mN_0} $ which satisfy $ (M.1) $ and $
(M.2)$ and $ p!  \subset M_p N_p $ ($ p! \prec M_p N_p $,
respectively). Then the following conditions are equivalent:
\begin{enumerate}
\item $f \in {\mathcal S}^{N_q} _{M_p} (\R^d)$
($ f \in \Sigma^{N_q} _{M_p}(\R^d), $ respectively).
\item  There exist constants $A,B\in \R^d,$ $ A,B >0$
(for every   $A,B\in \R^d,$ $ A,B >0$ respectively), and there exist
$ C>0$ such that
$$
\| e^{M(|Ax|)} \partial^{q}  f (x) \|_{L^\infty} \leq C B^{q} N_{|q|}, \quad\forall
p,q\in \N^d_0.
$$
\item There exist constants $A,B\in \R^d,$ $ A,B >0$
(for every  $A,B\in \R^d,$ $ A,B >0,$ respectively), and there exist
$ C>0$ such that
$$ \| x^{p}  f (x) \|_{L^\infty} \leq C A^{p} M_{|p|} \quad\mbox{and}\quad
\| \partial^{q}  f (x) \|_{L^\infty} \leq C B^{q} N_{|q|}, \quad\forall
p,q\in \N^d_0. $$
\item There exist constants $A,B\in \R^d,$ $ A,B >0$
(for every  $A,B\in \R^d,$ $ A,B >0,$ respectively), and there exist
$ C>0$ such that
$$
\| x^{p}  f (x) \|_{L^\infty} \leq C A^{p} M_{|p|} \quad\mbox{and}\quad
\| \omega^{q}  \hat{f} (\omega) \|_{L^\infty} \leq C B^{q} N_{|q|}, \quad\forall
p,q\in \N^d_0.
$$
\item There exist constants $A,B\in \R^d,$ $ A,B >0$
(for every  $A,B\in \R^d,$ $ A,B >0,$ respectively), such that
$$ \|  f(x) \;  \|_{L^\infty} < \infty \quad\mbox{and}\quad
\| \hat f (\o) \; e^{N(|B\o|)} \|_{L^\infty} < \infty, $$ where $
M(\cdot) $ and $ N(\cdot) $ are the associated functions for the
sequences $ (M_p)_{p \in \mN_0} $ and $ (N_q)_{q \in \mN_0} $
respectively.
\end{enumerate}
\end{theorem}

\begin{proof}
Theorem \ref{GS-characterization} is for the first time proved in \cite{CCK} and reinvented many times afterwards,
see e.g. \cite{GZ,KPP, PT1, CPRT1, NR}. As an illustration, and to give a flavor of the technique, we show {\em 1.} $ \Leftrightarrow $ {\em 2}. For the simplicity, we observe the Gevrey sequences  $ M_p = p!^{r} $ and  $ N_q = q!^{s}, $ $ p\in \mN_{0}$,
$ r, s  > 0$.

Recall, $f \in {\mathcal S}^{N_q} _{M_p} (\R^d) = {\mathcal S} ^{s} _{r}  (\mR^d)$ if and only if there exist constants
$  h,C > 0 $ such that
$$
\sup_{x \in \mR^d} \left |
x^{p} \partial^q  f (x) \right| \leq C h^{|p| +
|q|} p! ^{r} q!^{s}, \;\; \forall p, q \in
\mN_{0} ^d.
$$
To  avoid the use of inequalities related to multi-indices we consider $d= 1.$
Put $F_q (x) = \partial^q  f (x) / (h^{|q|} q!^{s})$. We have
$$
\sup_{x} h^{-|p|} p !^{-r}  | x|^{p} | F_q (x) | \leq C,
$$
so that
$$
\sup_{x}  h^{-|p|/r} p !^{-1}  | x|^{p/r} | F_q (x)  |^{1/r} \leq C^{1/r}
$$
uniformly in $p$.  Therefore
$$
 \sup_{x} \sum_{p \in \mathbb{N}_0}  ( \frac{(| x|h^{-1})^{1/r}}{2} )^{|p|}  p !^{-1}
 | F_q (x)  |^{1/r} \leq C^{1/r} \sum_{p \in \mathbb{N}_0} \frac{1}{2^{|p|}}.
$$
Put $A= h^{-1/r} 2^{-1} $, and conclude that there exist constants $A,B,C>0$ such that
$$
\left | \partial^q  f (x) \right| \leq  C B^{|q|}
q!^{s} e^{-A|x|^{1/r}}, \;\;
\forall x \in \mR, \;\;
\forall q \in \mN_{0},
$$
which gives {\em 2.}

Assume now that {\em 2.} holds.
Put $F_q (x) = \partial^q  f (x) / (C^{|q|} q!^{s}).$  Then,
$ \displaystyle \left | \partial^q  f (x) \right| \leq  C^{1+ |q|}
q!^{s} e^{-A|x|^{1/r}}$ for all $ x \in \mR$ implies the following chain of inclusions.
\begin{eqnarray*}
\left | F_q (x)  \right|^{1/r} e^{\frac{A}{r} |x|^{1/r} } < \infty
& \Rightarrow & \sup_{x} \sum_{p \in \mathbb{N}_0} \frac{1}{p !}
(\frac{A}{r})^p  |x|^{p/r} \left | F_q (x)  \right|^{1/r} < \infty \\
& \Rightarrow & \sup_{x}  \frac{1}{p !}  (\frac{A}{r})^p
 |x|^{p/r} \left | F_q (x)  \right|^{1/r} < \infty \\
& \Rightarrow & \sup_{x}  \frac{1}{p !^{r}}  (\frac{A}{r})^{r p}
 |x|^{p} \left | F_q (x)  \right| < \infty \\
& \Rightarrow &  \left  | x^{p}   \partial^q  f (x)   \right| \leq \tilde C \left ( (\frac{r }{A})^{r} \right )^{|p|} C^{|q|} p !^{r} q!^{s} \\
& \Rightarrow &  \left  | x^{p}  F_q (x)  \right| \leq \tilde Ch^{|p| +|q|} p !^{r} q!^{s},
\end{eqnarray*}
so that $ f \in {\mathcal S}^{N_q} _{M_p} (\R).$

The proof for $\Sigma^{N_q} _{M_p}(\R^d)$ is (almost) the same.
\qed
\end{proof}

By the above characterization $ {\mathcal F}  {\mathcal S}^{N_q} _{M_p}(\rd)  =  {\mathcal S}^{M_p} _{N_q} (\rd).  $
Observe that when  $M_p$ and $N_q$ are chosen to be Gevrey sequences, then $\cS^{1/2} _{1/2} (\rd)$ is the smallest non-empty Gelfand-Shilov space invariant
under the Fourier transform, see also Remark \ref{fine tuning} below. Theorem \ref{GS-characterization}
implies that $f\in\cS^{1/2} _{1/2}(\rd)$ if and only if
$f\in\cC^\infty(\rd)$ and there exist constants $h>0,k>0$ such that
\begin{equation} \label{condd1}
 \|f  e^{h |\cdot|^{2}} \|_{L^\infty} < \infty \quad\mbox{and}\quad
 \| \hat f  e^{k |\cdot|^{2}} \|_{L^\infty} < \infty.
\end{equation}
Therefore the Hermite functions
given by \eqref{herm1} belong to $\cS^{1/2} _{1/2} (\rd)$. This is an important fact when dealing with Gelfand-Shilov spaces,
cf. \cite{L06, P1}. We refer to \cite{Toft-2017} for the situation below the "critical exponent" $1/2$.

\begin{remark} \label{fine tuning}
Note that $  \Sigma^{1/2} _{1/2} (\rd)= \{ 0 \}$ and $  \Sigma^{s} _{s} (\rd)$ is
dense in the Schwartz space whenever $s>1/2$. One may consider a "fine tuning", that is the
spaces $\Sigma^{N_q} _{M_p} (\rd)$ such that
$$
\{ 0 \} = \Sigma^{1/2} _{1/2} (\rd) \hookrightarrow \Sigma^{N_q} _{M_p} (\rd)\hookrightarrow
{\mathcal S}^{N_q} _{M_p} (\rd) \hookrightarrow \Sigma^{s} _{s}(\rd), \;\;\; s > 1/2.
$$
For that reason, we define
sequences $ ( M_p )_{p \in \mN_0} $  and $ ( N_q )_{q \in \mN_0} $
by
\begin{equation} \label{big-seq-cond}
M_p := p!^{\frac{1}{2}} \prod_{k =0} ^p l_k =  p!^{\frac{1}{2}}
L_{p}, \;\;\; p \in \mN_0,\;\;\; N_q := q!^{\frac{1}{2}} \prod_{k
=0} ^q r_k =  q!^{\frac{1}{2}} R_q, \;\;\; q \in \mN_0
\end{equation}
where  $(r_p)_{p \in \mN_0}  $ and $(l_p)_{p \in \mN_0} $
are sequences of positive numbers monotonically increasing to infinity
such that \eqref{(M.2)} holds with the letter $s$ replaced by $r$ and $l$ respectively
and which satisfy: For
every $\alpha\in(0,1]$ and every $k>1$ so that $kp\in\mathbb{N},
p\in \mathbb{N},$
\begin{equation} \label{seq-cond}
\max\{(\frac{r_{kp}}{r_p})^2,(\frac{l_{kp}}{l_p})^2 \} \leq
k^\alpha, \;\;\; p\in\mathbb{N}.
\end{equation}
Then $ p! \prec M_p N_p$ and the sequences
 $ (R_p)_{p \in \mN_0} $ and $ (L_p)_{p \in \mN_0} $ ($R_p=r_1
\cdots r_p$,  $L_p=l_1  \cdots l_p, p\in \mathbb{N}$ $R_0=1,$ and $L_0=1$)
satisfy conditions $(M.1)$ and $(M.2) $. Moreover,
$$\max\{R_p,L_p\}\leq p!^{\alpha/2}, p\in\mathbb{N},$$
for every $\alpha\in (0,1]$.
(For $ p,q,k \in \N_{0} ^d $ we have
$ L_{|p|} =  \prod_{|k| \leq |p|} l_{|k|},$ and $ R_{|q|} =
\prod_{|k| \leq |q|} r_{|q|}.$)
Such sequences are used in the study of localization operators in the context of quasianalytic spaces in \cite{CPRT2}.
\end{remark}

\subsection{Kernel theorem} \label{sekcijajezgro}

For the study the action of linear operators it is convenient to use their kernels. In particular, when dealing with multilinear extensions of localization operators, the kernel theorem for Glefand-Shilov spaces appears to be a crucial tool, cf. \cite{Teof2019}.
Such theorems extend the famous Schwartz
kernel theorem (see \cite{Svarc,Trev}) to the spaces of ultradistributions. We refer to
\cite{Prang} for the proof in the case of non-quasianalytic Gelfand-Shilov spaces, and here we give a sketch of the proof
for a general case from \cite{Teof2015}. The only difference is that in quasianalytic case, the
density arguments from \cite{Prang} can not be used. Instead, we use arguments based on Hermite expansions in Gelfand-Shilov spaces, see \cite{L06, LCPT}.

\par

We need additional conditions for a
sequence of positive numbers $ (M_p)_{p \in \mN_0} $:

\noindent $\{N.1\} \;\;$
There exist positive constants $ A,H $ such
that
$$
 p!^{1/2} \leq A H^p  M_{p}, \;\; p \in \mN_0,
$$
and

\noindent $(N.1) \;\;$
For every $H>0$ there exists $ A >0$ such
that
$$
 p!^{1/2} \leq A H^p  M_{p}, \;\; p \in \mN_0.
$$
The  conditions $\{N.1\} $ and $(N.1)$
are taken from  \cite{LCP} where
they are called  {\em nontriviality conditions} for the spaces
${\mathcal S}^{M_p} _{M_p} (\R^{d})$ and $\Sigma^{M_p} _{M_p} (\R^{d})$ respectively.
In fact, the following lemma is proved in \cite{L06}.

\begin{lemma}
Let there be given a sequence of positive numbers $ (M_p)_{p \in \mN_0} $ which satisfies $ (M.1) $ and

$ (M.2)' \;\;$ There exist positive constants $ A,H $ such
that $ \displaystyle M_{p+1} \leq A H^p  M_{p}, \;\; p \in \mN_0. $

Then the following are equivalent:
\begin{enumerate} \label{ekvivHerm}
\item The Hermite functions are contained in
${\mathcal S}^{N_q} _{M_p} (\R^d)$
(in $ \Sigma^{N_q} _{M_p}(\R^d), $ respectively).
\item  $ (M_p)_{p \in \mN_0} $  satisfies $ \{N.1\} $
($ (M_p)_{p \in \mN_0} $  satisfies $ (N.1) $, respectively).
\item There are positive constants $ A,B$ and $H $ such
that
$$
 p!^{1/2} M_q \leq A B^{p+q} H^{p}  M_{p+q}, \;\; p,q \in \mN_0.
$$
(There is $B>0$ such that for every $H >0$ there exists $ A >0$ such
that
$$
 p!^{1/2} M_q \leq A B^{p+q} H^{p}  M_{p+q}, \;\; p,q \in \mN_0.
$$
\end{enumerate}
\end{lemma}

\par

We note that the condition $ (M.2)' $ is weaker than the condition
$ (M.2)$, and refer to  \cite[Remark 3.3]{L06} for the proof of Lemma \ref{ekvivHerm}.

\begin{theorem} \label{kernelteorema}
Let there be given a sequence of positive numbers $ (M_p)_{p \in
\mN_0} $ which satisfies $ (M.1) $,  $ (M.2)$ and $\{N.1\} $.
Then the following isomorphisms hold:
\begin{enumerate}
\item
$ {\mathcal S}^{M_p} _{M_p} (\R^{d_1+d_2}) \cong
\displaystyle {\mathcal S}^{M_p} _{M_p} (\R^{d_1}) \hat{\otimes} {\mathcal S}^{M_p} _{M_p} (\R^{d_2}) $
$$ \cong
\mathcal{L}_b ( ({\mathcal S}^{M_p} _{M_p})' (\R^{d_1}), {\mathcal S}^{M_p} _{M_p} (\R^{d_2})),
$$
\item
$
({\mathcal S}^{M_p} _{M_p})' (\R^{d_1+d_2}) \cong
\displaystyle
({\mathcal S}^{M_p} _{M_p})' (\R^{d_1}) \hat{\otimes} ({\mathcal S}^{M_p} _{M_p})' (\R^{d_2}) $
$$ \cong
\mathcal{L}_b ( {\mathcal S}^{M_p} _{M_p} (\R^{d_1}), ({\mathcal S}^{M_p} _{M_p})' (\R^{d_2})).
$$

\par

If the sequence $ (M_p)_{p \in \mN_0} $ satisfies $ (M.1) $,  $ (M.2)$ and $(N.1) $
instead, then the following isomorphisms hold:

\item $ \displaystyle
\Sigma^{M_p} _{M_p} (\R^{d_1+d_2}) \cong
\Sigma^{M_p} _{M_p} (\R^{d_1}) \hat{\otimes} \Sigma^{M_p} _{M_p} (\R^{d_2}) $
$$
\cong \mathcal{L}_b ( (\Sigma^{M_p} _{M_p})' (\R^{d_1}), \Sigma^{M_p} _{M_p} (\R^{d_2})),
$$
\item
$ \displaystyle
(\Sigma^{M_p} _{M_p})' (\R^{d_1+d_2}) \cong
(\Sigma^{M_p} _{M_p})' (\R^{d_1}) \hat{\otimes} (\Sigma^{M_p} _{M_p})' (\R^{d_2}) $
$$ \cong
\mathcal{L}_b ( \Sigma^{M_p} _{M_p} (\R^{d_1}), (\Sigma^{M_p} _{M_p})' (\R^{d_2})).
$$
\end{enumerate}
\end{theorem}

\begin{proof}
By \cite[Remark 3.3]{L06} it follows that  $\{N.1\} $ is equivalent to
$$
H_k (x) \in {\mathcal S}^{M_p} _{M_p} (\R^{d_1}),  \;\; x \in\bR^{d_1},  k \in\bN^{d_1} _0,
\;\; \text{and} \;\;
H_l (y) \in {\mathcal S}^{M_p} _{M_p} (\R^{d_2}),  \;\; y \in\bR^{d_2},  l \in\bN^{d_2} _0,
$$
where $H_k (x) $ and $ H_l (y) $ are the Hermite functions given by \eqref{herm1}.
Now, by representation theorems  from \cite{L06} and \cite{LCP} and the fact that
$ H_{(k,l)} (x,y) \in {\mathcal S}^{M_p} _{M_p} (\R^{d_1+d_2}) $,  $(x,y) \in\bR^{d_1 + d_2}, $ $ (k,l)\in \bN^{d_1+d_2} _0,$
it follows that
$ {\mathcal S}^{M_p} _{M_p} (\R^{d_1}) {\otimes} {\mathcal S}^{M_p} _{M_p} (\R^{d_2}) $ is dense in
${\mathcal S}^{M_p} _{M_p} (\R^{d_1+d_2})$.

In order to obtain the isomorphism $ {\mathcal S}^{M_p} _{M_p} (\R^{d_1+d_2}) \cong
{\mathcal S}^{M_p} _{M_p} (\R^{d_1}) \hat{\otimes} {\mathcal S}^{M_p} _{M_p} (\R^{d_2}) $
it is sufficient to
prove that $ {\mathcal S}^{M_p} _{M_p} (\R^{d_1+d_2})$ induces
the $\pi = \epsilon$  topology on the product space
$ {\mathcal S}^{M_p} _{M_p} (\R^{d_1}) \hat{\otimes} {\mathcal S}^{M_p} _{M_p} (\R^{d_2})$.
The topologies $ {\mathcal S}^{M_p} _{M_p} (\R^{d_1}) {\otimes}_\pi {\mathcal S}^{M_p} _{M_p} (\R^{d_2})$
and $ {\mathcal S}^{M_p} _{M_p} (\R^{d_1}) {\otimes}_\epsilon  {\mathcal S}^{M_p} _{M_p} (\R^{d_2})$
coincide since $ {\mathcal S}^{M_p} _{M_p} (\R^{d}) $ is a nuclear space (see e.g. \cite{LCP}).
We  refer to \cite[Chapter 43]{Trev} for the definition and basic facts on the $\pi $ and $\epsilon$
topologies.

The idea of the proof is to show that the topology $\pi$ on $ {\mathcal S}^{M_p} _{M_p} (\R^{d_1}) {\otimes}  {\mathcal S}^{M_p} _{M_p} (\R^{d_2})$ is stronger than the  one induced from $ {\mathcal S}^{M_p} _{M_p} (\R^{d_1+d_2})$,
and that the $\epsilon$ topology on $ {\mathcal S}^{M_p} _{M_p} (\R^{d_1}) {\otimes}  {\mathcal S}^{M_p} _{M_p} (\R^{d_2})$  is weaker than the induced one.
This will imply  the isomorphism
$$
{\mathcal S}^{M_p} _{M_p} (\R^{d_1+d_2}) \cong
{\mathcal S}^{M_p} _{M_p} (\R^{d_1}) \hat{\otimes} {\mathcal S}^{M_p} _{M_p} (\R^{d_2}).
$$

That proof is quite technical, and we give here only a sketch, cf.  \cite{Teof2015} for details.
For the $\pi$ topology, we use a convenient separately continuous bilinear mapping
which implies the continuity of the inclusion
$ {\mathcal S}^{M_p} _{M_p} (\R^{d_1}) {\otimes}_\pi {\mathcal S}^{M_p} _{M_p} (\R^{d_2}) \rightarrow
{\mathcal S}^{M_p} _{M_p} (\R^{d_1+d_2})$.

For the $\epsilon$ topology, the proof is more technically involved. Namely, for a given equicontinuous subsets $ A' \subset {\mathcal S}^{M_p} _{M_p} (\R^{d_1}) $ and
$ B' \subset {\mathcal S}^{M_p} _{M_p} (\R^{d_2}) $, we use a particularly chosen family of norms which defines a topology
equivalent to the one given by Definition \ref{GSoftypeS} to estimate
$ | \langle F_x \otimes \tilde F_y, \Phi (x,y) \rangle | $, $ F_x \in A' $ and $
\tilde F_y \in B'$.

Next, we use the fact that
$({\mathcal S}^{M_p} _{M_p})' (\R^{d_1})$  and $ {\mathcal S}^{M_p} _{M_p} (\R^{d_2})$ are complete
and that $({\mathcal S}^{M_p} _{M_p})' (\R^{d_1})$ is barreled. Moreover,
$ {\mathcal S}^{M_p} _{M_p} (\R^{d_1})$ is nuclear and complete, so that \cite[Proposition 50.5]{Trev} implies that
$ \mathcal{L}_b ( ({\mathcal S}^{M_p} _{M_p})' (\R^{d_1}), {\mathcal S}^{M_p} _{M_p} (\R^{d_2}))$ is
complete and that
$$ \displaystyle
{\mathcal S}^{M_p} _{M_p} (\R^{d_1}) \hat{\otimes} {\mathcal S}^{M_p} _{M_p} (\R^{d_2}) \cong
\mathcal{L}_b ( ({\mathcal S}^{M_p} _{M_p})' (\R^{d_1}), {\mathcal S}^{M_p} _{M_p} (\R^{d_2})).
$$
This proves 1) and we leave the other claims to the reader, see also \cite{Prang}.
\qed
\end{proof}

The isomorphisms in Theorem \ref{kernelteorema} 2)  tells us that for a given kernel-distribution
$ k(x,y) $ on $\R^{d_1+d_2}$ we may associate a continuous linear mapping
$k$ of $  {\mathcal S}^{M_p} _{M_p} (\R^{d_2})$ into
$ ({\mathcal S}^{M_p} _{M_p})' (\R^{d_1})$ as follows:
$$
\langle k_\varphi, \phi \rangle = \langle k(x,y), \phi(x) \varphi(y) \rangle, \;\;\; \phi \in
 {\mathcal S}^{M_p} _{M_p} (\R^{d_1}),
$$
which is commonly written as $k_\varphi (\cdot) = \int k(\cdot, y )\varphi(y) dy.$ By Theorem \ref{kernelteorema} b)
it follows that the correspondence between $k(x,y) $ and $k$ is  an isomorphism.
Note also that the transpose $^t k$ of the mapping $k$ is given by
$ \displaystyle ^t k_\phi (\cdot) = \int k(x,\cdot )\phi(x) dx.$

By the above isomorphisms we conclude that for any  continuous and linear mapping
between  $ {\mathcal S}^{M_p} _{M_p } (\mathbb{R}^{2d}) $ to
$ ({\mathcal S}^{M_p} _{M_p })'  (\mathbb{R}^{2d})$ one can assign a uniquely determined kernel
with the above mentioned properties.
We will use this fact in the proof of Theorem \ref{Weyl connection lemma}, and
refer to \cite[Chapter 52]{Trev} for applications of kernel theorems in linear partial differential equations.

\begin{remark}
The choice of the Fourier transform invariant spaces of the form $  {\mathcal S}^{M_p} _{M_p} (\R^{d})$
in Theorem \ref{kernelteorema} is not accidental. We refer to \cite{GLPR} where it is proved that
if the Hermite expansion $ \sum_{k \in \bN^{d}} a_k H_k (x)$   converges to $f$
($a_k$ are the Hermite coefficients of $f$)
in the sense of
$ {\mathcal S}^{s} _{r} (\R^{d})$ ($ {\Sigma}^{s} _{r} (\R^{d})$, respectively), $ r<s,$
then it belongs to $ {\mathcal S}^{r} _{r} (\R^{d})$ ($ {\Sigma}^{r} _{r} (\R^{d})$, respectively).
\end{remark}

\subsection{Time-frequency analysis of Gelfand-Shilov spaces} \label{Sec3}

In this section we extend the action of the time-frequency representations from Section \ref{Grossmann-Royer}
to the Gelfand-Shilov spaces and their dual spaces.
To that end we observe the following modification of Definition \ref{GSoftypeS}.


\begin{definition} \label{GSoftypeS-2d}
Let there be given sequences of positive numbers
$ (M_p)_{p \in \mN_0} $, $ (N_q)_{q \in \mN_0} $,
$ (\tilde M_p)_{p \in \mN_0} $, $ (\tilde N_q)_{q \in \mN_0} $
which satisfy $ (M.1) $ and $ (M.2).$ We define
$ \cS^{N_{q} , \tilde N_{q} ,B} _{M_{p}, \tilde M_{p}, A} (\R^{2d}) $ to be the set of smooth functions
$ f \in C^\infty(\R^{2d}) $ such that
$$
\| x^{\alpha_1} \o^{\alpha_2}\partial^{\beta_1} _x \partial^{\beta_2} _\o f  \|_{L^\infty} \leq
C A^{|\alpha_1 + \alpha_2|} M_{|\alpha_1|} \tilde M_{|\alpha_2|}
B^{|\beta_1 + \beta_2|} N_{|\beta_1|} \tilde N_{|\beta_2|},
$$
$$
\forall \alpha_1, \alpha_2, \beta_1, \beta_2 \in \N_{0} ^d \},
$$
and for some $A, B, C >0.$ {\em Gelfand-Shilov spaces} are projective and inductive limits of
$ \cS^{N_{q} , \tilde N_{q} ,B} _{M_{p}, \tilde M_{p}, A} (\R^{2d}) $:
$$
\Sigma^{N_q, \tilde N_q} _{M_p, \tilde M_p} (\R^{2d}) : = {\rm proj} \lim_{A>0, B>0}
\cS^{N_{q} , \tilde N_{q} ,B} _{M_{p}, \tilde M_{p}, A} (\R^{2d});
$$
$$
\cS^{N_q, \tilde N_q} _{M_p, \tilde M_p} (\R^{2d})  : = {\rm ind} \lim_{A>0, B>0}
\cS^{N_{q} , \tilde N_{q} ,B} _{M_{p}, \tilde M_{p}, A} (\R^{2d}).
$$
\end{definition}

Clearly, the corresponding dual spaces are given by
$$
(\Sigma^{N_q, \tilde N_q} _{M_p, \tilde M_p})' (\R^{2d}) : = {\rm ind} \lim_{A>0, B>0}
(\cS^{N_{q} , \tilde N_{q} ,B} _{M_{p}, \tilde M_{p}, A})' (\R^{2d});
$$
$$
(\cS^{N_q, \tilde N_q} _{M_p, \tilde M_p})' (\R^{2d})  : = {\rm proj} \lim_{A>0, B>0}
(\cS^{N_{q} , \tilde N_{q} ,B} _{M_{p}, \tilde M_{p}, A})' (\R^{2d}).
$$

By Theorem \ref{GS-characterization},  the Fourier transform is a homeomorphism from
$ \Sigma^{N_q, \tilde N_q} _{M_p, \tilde M_p} (\R^{2d}) $
to $\Sigma_{N_q, \tilde N_q} ^{M_p, \tilde M_p} (\R^{2d}) $
and, if $\Fur_1 f$ denotes the partial Fourier transform of $f(x,\o)$
with respect to the $x$ variable, and if $\Fur_2 f$
denotes the partial Fourier transform of $f(x,\o)$
with respect to the $\o$ variable, then
$ \Fur_1 $ and $ \Fur_2$ are homeomorphisms from
$ \Sigma^{N_q, \tilde N_q} _{M_p, \tilde M_p} (\R^{2d}) $
to $\Sigma_{M_p, \tilde N_q} ^{N_q, \tilde M_p} (\R^{2d}) $ and
$ \Sigma^{N_q, \tilde M_p} _{M_p, \tilde N_q} (\R^{2d}) $, respectively.
Similar facts hold when $ \Sigma^{N_q, \tilde N_q} _{M_p, \tilde M_p} (\R^{2d}) $ is replaced by
$\cS^{N_q, \tilde N_q} _{M_p, \tilde M_p} (\R^{2d})$, $
(\Sigma^{N_q, \tilde N_q} _{M_p, \tilde M_p})' (\R^{2d})$ or
$ (\cS^{N_q, \tilde N_q} _{M_p, \tilde M_p})' (\R^{2d})$.

When $M_p = \tilde M_p $ and $ N_q = \tilde N_q $ we use usual abbreviated notation:
$\cS^{N_q} _{M_p}  (\R^{2d}) = \cS^{N_q, \tilde N_q} _{M_p, \tilde M_p} (\R^{2d})$
and similarly for other spaces.

\par

Let $ (M_p)_{p \in \mN_0} $ satisfy $ (M.1) $, $ (M.2)$
and $ \{ N.1\} $ ($ (N.1) $, respectively).
For any given  $f, g \in \cS ^{M_p}_{M_p} (\rd )$
($f, g \in \Sigma ^{M_p}_{M_p} (\rd )$, respectively)  the
Grossmann-Royer transform of $f$ and $g$ is given by
\eqref{GRT}, i.e.
$$
R_g f (x,\omega) = \int e^{4\pi i  \omega(t-x)} f(2x- t) \overline{ g(t)} dt, \;\;\; x,\omega \in \mathbb{R}^d,
$$
and the definition can be extended to  $f \in (\cS ^{M_p}_{M_p})' (\rd ) $
($f \in (\Sigma ^{M_p}_{M_p})' (\rd )$, respectively) by duality.

Similarly, another \tf\, representations, the \stft\ $  V_g $,
the cross-Wigner distribution $W(f,g)$, and the cross-ambiguity function $ A (f,g)$
given by \eqref{GT}, \eqref{WD} and \eqref{FWT},
can be extended to   $f \in (\cS ^{M_p}_{M_p})' (\rd ) $
($f \in (\Sigma ^{M_p}_{M_p})' (\rd )$, respectively) when $g \in \cS ^{M_p}_{M_p} (\rd )$
($g \in \Sigma ^{M_p}_{M_p} (\rd )$, respectively).

The following theorem and its variations is a folklore,
in particular in the framework of the duality between $\cS (\R^{2d})$ and $\cS^{'} (\R^{2d})$.
For Gelfand-Shilov spaces we refer to e.g.  \cite{GZ, T2, Teof2015, Toft-2012}.

\begin{theorem} \label{nec-suf-cond}
Let there be given  sequences $ ( M_p )_{p \in \mN_0} $ and
$ ( N_q )_{q \in \mN_0} $ which satisfy
(M.1), (M.2) and $ \{ N.1\} $, and let $ TFR (f,g) \in \{ R_g f, V_g f, W (f,g), A\}.$
If $f,g \in \cS^{N_q} _{M_p} (\R^d)$,
then $ TFR (f,g)   \in \cS^{N_q, M_p} _{M_p, N_q} (\rdd) $
and extends uniquely to a continuous map from
$ (\cS^{N_q} _{M_p})' (\R^d)\times  (\cS^{M_p} _{N_q})' (\R^d) $
into $ (\cS^{N_q,M_p} _{M_p, N_q})' (\R^{2d}).$

Conversely, if  $ TFR (f,g) \in \cS^{N_q,M_p} _{M_p,N_q} (\rdd) $
then $f,g \in \cS^{N_q} _{M_p} (\R^d).$

Let the sequences $ ( M_p )_{p \in \mN_0} $ and
$ ( N_q )_{q \in \mN_0} $ satisfy
(M.1), (M.2) and $ ( N.1) $ instead.
If $f,g \in \Sigma^{N_q} _{M_p} (\R^d), $
then $ TFR(f,g)   \in \Sigma^{N_q, M_p} _{M_p, N_q}  (\rdd) $
and extends  uniquely to a continuous  map from $ (\Sigma^{N_q} _{M_p})' (\R^d) \times  (\Sigma^{M_p} _{N_q})' (\R^d) $
into $ (\Sigma^{N_q,M_p} _{M_p, N_q})' (\R^{2d}).$

Conversely, if  $ TFR (f,g) \in \Sigma^{N_q,M_p} _{M_p, N_q} (\rdd)$
then $f,g \in \Sigma^{N_q} _{M_p} (\R^d).$
\end{theorem}

\begin{proof}
Since Gelfand-Shilov spaces
are closed under reflections, dilations and modulations, by Lemma \ref{GRandRelatives}
it is enough to give the proof for the Grossmann-Royer transform, and the same conclusion holds for other \tf\, representations.
But the proof is essentially the same as the proof of Proposition \ref{svojstva} {\em 6.}
We recall \eqref{GRrewritten}:
$$
R_g f (x,\omega) = (\mathcal{F}_2 \circ \tau^*) f \otimes \overline{ g} (x,\omega), \;\;\;x,\omega \in \mathbb{R}^d,
$$
Since the pullback operator $ \tau^*$ is a continuous bijection on
$ \cS^{N_q, N_q} _{M_p, M_p} (\R^{2d})$, and
$ \mathcal{F}_2 $ is a continuous bijection between $ \cS^{N_{q}, N_q} _{M_{p}, M_p} (\R^{2d})$
and  $ \cS^{N_{q}, M_p} _{M_{p}, N_q} (\R^{2d})$,
we obtain
$$
R_g f \in  \cS^{N_{q}, M_p} _{M_p,N_q}  (\rdd) \Leftrightarrow
f \otimes \overline{g}  \in  \cS^{N_{q}, N_q} _{M_p, M_p}  (\rdd) \Leftrightarrow f,g \in
\cS^{N_q} _{M_p} (\R^d).
$$
Moreover, $ R_g f$ can be extended to a map from
$ (\cS^{N_q} _{M_p})' (\R^d)\times  (\cS^{M_p} _{N_q})' (\R^d) $
into $ (\cS^{N_q, M_p} _{M_p, N_q})' (\R^{2d})$ by duality.

To prove that
$ R_g f   \in \cS^{N_q, M_p} _{M_p, N_q} (\rdd) $ when
$f,g \in \cS^{N_q} _{M_p} (\R^d)$, we could also perform direct calculations based on the  following observations.

Asume that
$g \in \cS^{\tilde N_{q}} _{\tilde M_{p}} (\R^d) $
where $ (\tilde M_p )_{p \in \mN_0} $ and
$ (\tilde N_q )_{q \in \mN_0} $ satisfy (M.1), (M.2),
$\tilde M_p \subset M_p$ and $ \tilde N_q \subset N_q$, which is  a slightly more general situation.
Than $ f(x)\otimes g(t) \in  \cS^{N_{q}, \tilde N_q} _{M_{p}, \tilde M_p} (\R^d \times \R^d ).$

Put $ \varphi (x,t) := f(2x -t) g (t)$.
If we show
\begin{equation} \label{varphi po x}
\sup_{x,t \in \mathbb{R}^d}  | x^{\alpha}  t^\beta  \varphi (x,t) | \leq C h^{|\alpha| + |\beta|}
M_{|\alpha|} M_{|\beta|},
\end{equation}
and
\begin{equation} \label{varphi po ksi}
\sup_{x,t \in \mathbb{R}^d}  |  \partial^\alpha _x \partial^\beta _t  \varphi (x,t) |
\leq C k^{|\alpha|+ |\beta|} N_{|\alpha|}  N_{|\beta|}
\end{equation}
for some  $ h,k > 0,$ then by Theorem \ref{GS-characterization} it follows that
$ \varphi \in \cS^{N_{q}, N_q} _{M_{p}, M_p} (\R^{2d} ).$

The first inequality easily follows from assumptions on $f$ and $g$ and a change of variables:
$$
\sup_{x,t \in \mathbb{R}^d}  | x^{\alpha}  t^\beta  f(2x-t) g (t) |
\leq  2^{-|\alpha|} \sup_{y,t \in \mathbb{R}^d}  | (y+ t)^{\alpha}  t^\beta  f(y) g(t)|,
$$
and \eqref{varphi po x} follows from the assumptions on $f$, $g$ and $\tilde M_p \subset M_p$.
To prove \eqref{varphi po ksi}, we use the Leibniz formula which gives
\begin{eqnarray*}
 |\partial^\alpha _x \partial^\beta _t  \varphi (x,t)|
& =&
 | \sum_{\gamma \leq \beta}
  {\beta \choose \gamma}
 \frac{1}{2^{|\alpha| + |\beta|} }
 \partial^{\alpha} _x \partial^{\gamma} _t  f (2x - t)
 \partial^{\beta - \gamma} _t g  (t)|
\\
& \leq & C_{\alpha, \beta} \sup_{x,t \in \mathbb{R}^d}
| \partial^{\alpha} _x \partial^{\gamma} _t  f (2x - t)
\partial^{\beta - \gamma} _t g  (t)|.
\end{eqnarray*}
Next we use $ \tilde N_q \subset N_q$ and conditions (M.1) and (M.2) applied to the sequence $(N_q)$
to obtain \eqref{varphi po ksi}.
Therefore, $ \varphi \in \cS^{N_{q}, N_q} _{M_{p}, M_p} (\R^{2d} ).$

Now, the partial inverse Fourier transform of $\varphi$ with respect to the second variable
is continuous bijection between $ \cS^{N_{q}, N_q} _{M_{p}, M_p} (\R^{2d})$
and  $ \cS^{N_{q}, M_p} _{M_{p}, N_q} (\R^{2d})$, and those spaces are closed under  dilations, so that
$$
R_g f (x,\o) = e^{-4\pi i \o x}  \int e^{2\pi i \o (2t) } \varphi (x,t) dt \in
\cS^{N_{q}, M_p} _{M_p N_q} (\R^{2d})
$$
if and only if $ \varphi \in \cS^{N_{q}, N_q} _{M_{p}, M_p} (\R^{2d}). $
The extension  to a map from
$ (\cS^{N_q} _{M_p})' (\R^d)\times  (\cS^{M_p} _{N_q})' (\R^d) $
into $ (\cS^{N_q, M_p} _{M_p, N_q})' (\R^{2d})$ is straightforward.
\qed
\end{proof}


\section{Modulation Spaces} \label{Modulation Spaces}

The modulation space norms traditionally  measure
the joint time-frequency distribution of $f\in \sch '$,
we refer, for instance, to \cite{F1}, \cite[Ch.~11-13]{Gro} and
the original literature quoted there for various properties and applications.
It is usually sufficient to observe modulation spaces with weights which admit at most polynomial growth at infinity.
However the study of ultra-distributions requires a more general approach that includes the weights of exponential or even superexponential growth, cf. \cite{CPRT1, Toft-2017}.
Note that the general approach introduced already in  \cite{F1}
includes the weights of sub-exponential growth (see (\ref{BDweight})).
We refer to \cite{FG-Atom1, FG-Atom2}
for related but even more general constructions, based on the general theory of coorbit spaces.

\emph{Weight Functions.}  In the sequel $v$ will always be a
continuous, positive,  even, submultiplicative   function
(submultiplicative weight), i.e., $v(0)=1$, $v(z) =
v(-z)$, and $ v(z_1+z_2)\leq v(z_1)v(z_2)$, for all $z,
z_1,z_2\in\Renn.$ Moreover, $v$ is assumed to be even in each group of coordinates,  that is, $ v (x, \o) = v (-\o, x)= v(-x,\o), $ for any $\phas\in\rdd$. Submultipliciativity implies that $v(z)$ is \emph{dominated} by an exponential function, i.e.
\begin{equation} \label{weight}
 \exists\, C, k>0 \quad \mbox{such\, that}\quad  v(z) \leq C e^{k \|z\|},\quad z\in \rdd,
\end{equation}
and   $\|z\|$ is  the Euclidean norm of $z\in \rdd$.
For example, every weight of the form
\begin{equation} \label{BDweight}
v(z) =   e^{s\|z\|^b} (1+\|z\|)^a \log ^r(e+\|z\|)
\end{equation}
 for parameters $a,r,s\geq 0$, $0\leq b \leq 1$ satisfies the
above conditions.

 For our investigation of localization
operators we will mostly use the exponential weights   defined by
\begin{eqnarray}
   w_s (z)&=&w_s\phas = e^{s\|\phas\|},\quad
   z=(x,\o)\in\Renn \,, \label{eqc1}\\
 \tau_s(z)&=& \tau _s \phas = e^{s\|\o\|}  \,\label{eqc2}.
\end{eqnarray}
Notice  that arguing on $\bR ^{4d}$ we may read
\begin{equation} \label{tau}
\tau _s (z,\zeta ) = w_s (\zeta )   \quad \quad z,\zeta \in \Renn \,,
\end{equation}
which will be used in the sequel.

Associated to every submultiplicative weight we consider the class of
so-called  {\it
  v-moderate} weights $\cM _v$. A  positive, even
weight function  $m$ on $\Renn$ belongs to $\cM _v$ if it  satisfies
the condition
$$
 m(z_1+z_2)\leq Cv(z_1)m(z_2)  \quad   \forall z_1,z_2\in\Renn \, .
$$
 We note that this definition implies that
$\frac{1}{v} \lesssim m \lesssim v $,  $m \neq 0$ everywhere, and that
$1/m \in \cM _v$.

Depending on the
growth of the weight function $m$, different Gelfand-Shilov
classes may be chosen as fitting test function spaces for \modsp
s, see \cite{CPRT1,T2,Toft-2017}. The widest class of weights allowing to
define \modsp s is the weight class $\cN$. A weight function  $m$
on $\rdd$ belongs to $\cN$  if it is a continuous, positive
function such that
\begin{equation}\label{s12}
m(z)=o(e^{c z^2}),\,\quad\mbox{for}\,\,|z|\rightarrow\infty,\quad
\forall c>0,
\end{equation}
with $z\in\rdd$. For instance, every function $m(z)=e^{s |z|^b}$,
with $s>0$ and $0\leq b<2$,  is in $\cN$. Thus, the weight $m$ may
grow faster than exponentially at infinity. For example, the choice $ m \in \cN \setminus \cM _v$
is related to the spaces of  quasianalytic functions, \cite{CPRT2}.
We notice that  there
is a limit in enlarging the weight class for \modsp s, imposed by
Hardy's theorem:
if $m(z)\geq C e^{c z^2}$, for some $c>\pi/2$, then the corresponding \modsp s  are trivial \cite{GZ01}.
We refer to \cite{Gro07} for a
survey on the most important types of weights commonly used in time-frequency analysis.

\begin{definition}\label{defmodnorm}
Let $m\in \cN$, and $g$ a non-zero \emph{window} function in $\cS^{1/2}_{1/2}(\rd)$. For
$1\leq p,q\leq \infty$ the {\it  modulation space} $M^{p,q}_m(\Ren)$ consists of all
$f\in (\cS^{1/2}_{1/2})' (\rd)$
such that $V_gf\in L^{p,q}_m(\Renn )$
(weighted mixed-norm spaces). The norm on $M^{p,q}_m$ is
$$
\|f\|_{M^{p,q}_m}=\|V_{g}f\|_{L^{p,q}_m}=\left(\int_{\Ren}
  \left(\int_{\Ren}|V_{g} f(x,\o)|^pm(x,\o)^p\,
    dx\right)^{q/p}d\o\right)^{1/q}
$$
(with obvious changes if either $p=\infty$ or $q=\infty$). If
$p,q< \infty $, the \modsp\ $\Mmpq $ is the norm completion of
$\cS^{1/2}_{1/2}$ in the $\Mmpq $-norm. If $p=\infty $ or
$q=\infty$, then $\Mmpq $ is the completion of $\cS^{1/2}_{1/2}$
in the weak$^*$ topology.
\end{definition}

In this paper we restrict ourselves to   $v$-moderate weights $\cM _v$.
Then, for $f,g \in \cS^{(1)} (\rd )$ $ = \Sigma _1 ^1  (\rd ) $
the above integral is convergent thanks to Theorem \ref{nec-suf-cond}.
Namely, in  view of \eqref{weight}, for a given $ m\in\cM _v$ there exist $l>0$ such that
$ m (x,\o) \leq C e^{l \|\phas\|}$ and therefore
\begin{eqnarray*}
&& \left|\int_{\Ren}
  \left ( \int_{\Ren}|V_gf(x,\o)|^p m(x,\o)^p\,
    dx\right)^{q/p}d\o\right|\\
   && \quad\quad\quad\quad
     \leq
    C \left|\int_{\Ren}
  \left( \int_{\Ren}|V_gf(x,\o)|^p e^{l p\|\phas\|}\,
    dx\right) ^{q/p}  d\o\right| < \infty
\end{eqnarray*}
since by  Theorems \ref{nec-suf-cond} and Theorem \ref{GS-characterization}
we have $ |V_gf(x,\o)| < C e^{-s \|\phas\| } $ for every $s > 0.$
This implies $ \cS ^{(1)} \subset M^{p,q}_m.$

In particular, when $m$ is a polynomial weight of the form $m (x, \omega) = \langle  x \rangle ^t
\langle \omega \rangle ^s$ we will use the notation
$M^{p,q}_{s,t}(\mathbb{R}^d)$ for the modulation spaces which consists of all
$f\in \mathcal{S}'(\mathbb{R}^d)$ such that
$$
\| f \|_{M^{p,q}_{s,t}} \equiv \left  ( \int _{\mathbb{R}^d} \left ( \int _{\mathbb{R}^d}
|V_\phi f(x,\omega )\langle  x \rangle ^t
\langle \omega \rangle ^s|^p\, dx  \right )^{q/p}d\omega  \right )^{1/q}<\infty
$$
(with obvious interpretation of the integrals when $p=\infty$ or $q=\infty$).

If $p=q$, we write $M^p_m$ instead of $M^{p,p}_m$, and if $m(z)\equiv 1$ on $\Renn$, then we write $M^{p,q}$ and $M^p$ for $M^{p,q}_m$ and $M^{p,p}_m$, and so on.

In the next proposition we show that  $\Mmpq (\Ren )$ are  Banach spaces
whose definition is independent of the choice of the window
$g \in M^1_{v} \setminus \{ 0\}$.
In order to do so, we need the adjoint of the short-time Fourier transform.

For given window $ g \in  \mathcal{S} ^{(1)} $ and a
function $ F (x,\xi) \in L^{p,q} _m (\R^{2d})$ we define $ V^* _g F $ by
$$
\langle V^* _g F, f \rangle := \langle F, V_g f \rangle,
$$
whenever the duality is well defined.

In our context, \cite[Proposition 11.3.2]{Gro} can be rewritten as follows.

\begin{proposition} \label{emjedanve}
Fix $m \in  \cM _v$ and $ g, \psi \in  \mathcal{S} ^{(1)},$ with $\langle g, \psi \rangle\not= 0$. Then
\begin{enumerate}
\item $ V^* _g :  L^{p,q} _m (\R^{2d}) \rightarrow  M^{p,q} _m (\R^{d}), $ and
\begin{equation}  \label{vstar}
\|  V^* _g F  \|_{ M^{p,q} _m } \leq C \| V_\psi g \|_{ L^{1} _v }   \| F \|_{L^{p,q} _m}.
\end{equation}
\item The inversion formula holds: $ I_{ M^{p,q} _m }  = \langle g, \psi \rangle^{-1}
  V^* _g   V _\psi,$ where  $ I_{ M^{p,q} _m } $ stands for the identity operator.
\item   $\Mmpq (\Ren )$ are  Banach spaces
whose definition is independent on the choice of
$ g \in \mathcal{S} ^{(1)} \setminus \{ 0 \} $.
\item The space of admissible windows can be extended from $ {\mathcal S} ^{(1)} $ to $M^1 _v .$
\end{enumerate}
\end{proposition}

\begin{proof} We refer to \cite{CPRT1} for the proof which is based on the proof of
\cite[Proposition 11.3.2.]{Gro}. Note that for  {\em 4.} we need the density of $ \cS ^{(1)} $ in $ M^{p,q}_m.$
This fact is not obvious, we refer to  \cite{Elena07} for the proof.
Then {\em 4.}) follows by using standard arguments of \cite[Theorem 11.3.7]{Gro}.

Note that this result actually implies that Definition \ref{defmodnorm}
coincides with the usual definition of modulation spaces with weights of
polynomial and sub-exponential growth (see, for example \cite{CG02, F1, Gro, PT2}).
\qed
\end{proof}

The following theorem lists some basic properties of modulation spaces.
We refer to \cite{F1, Gro, GZ, PT1, T3, Toft-2012} for the proof.

\begin{theorem} \label{modproerties}
Let $p,q,p_j,q_j\in [1,\infty ]$ and $s,t,s_j,t_j\in \mathbb{R}$, $j=1,2$. Then:
\begin{enumerate}
\item $M^{p,q}_{s,t}(\mathbb{R}^d)$ are Banach spaces, independent of the choice of
$\phi \in \mathcal{S}(\mathbb{R}^d) \setminus 0$;

\item if  $p_1\le p_2$, $q_1\le q_2$, $s_2\le s_1$ and
$t_2\le t_1$, then
$$
\mathcal{S}(\mathbb{R}^d)\subseteq M^{p_1,q_1}_{s_1,t_1}(\mathbb{R}^d)
\subseteq M^{p_2,q_2}_{s_2,t_2}(\mathbb{R}^d)\subseteq
\mathcal{ S}'(\mathbb{R}^d);
$$

\item $ \displaystyle
\cap _{s,t} M^{p,q}_{s,t}(\mathbb{R}^d)=\mathcal{ S}(\mathbb{R}^d),
\quad
\cup _{s,t}M^{p,q}_{s,t}(\mathbb{R}^d)=\mathcal{ S}'(\mathbb{R}^d); $
\item Let $ 1\leq p, q \leq  \infty, $ and let $ w_s $ be given by \eqref{eqc1}.
Then
$$
\Sigma_1 ^1 (\mathbb{R}^d)=  {\mathcal S}^{(1)} (\mathbb{R}^d)=  \bigcap _{s \geq  0} M _{w_{s}} ^{p,q} (\mathbb{R}^d),\;\;\;
(\Sigma_1 ^1)' (\mathbb{R}^d)
= \bigcup _{s \geq 0} M _{1/w_{s}} ^{p,q} (\mathbb{R}^d),
$$
$$
{\mathcal S}_1 ^{1} (\mathbb{R}^d) = {\mathcal S}^{\{1\}} (\mathbb{R}^d) =  \bigcup _{s >  0} M _{w_{s}} ^{p,q} (\mathbb{R}^d),\;\;\;
({\mathcal S}_1 ^{1})' (\mathbb{R}^d)
= \bigcap _{s > 0} M _{1/w_{s}} ^{p,q} (\mathbb{R}^d).
$$
\item For   $p,q\in [1,\infty )$, the dual of $ M^{p,q}_{s,t}(\mathbb{R}^d)$ is
$ M^{p',q'}_{-s,-t}(\mathbb{R}^d),$ where $ \frac{1}{p} +  \frac{1}{p'} $ $ =
 \frac{1}{q} +  \frac{1}{q'} $ $ =1.$
\end{enumerate}
\end{theorem}

\begin{remark} \label{GSandmod}
Alternatively,  $ {\mathcal S}_1 ^{1} (\mathbb{R}^d)$ can also be viewed as a projective limit
(and its dual space as an inductive limit) of modulation spaces as follows:
$$
{\mathcal S}_1 ^{1} (\mathbb{R}^d) =  \bigcap _{m  \in \cap \cM _{w_s}} M _{m} ^{p,q} (\mathbb{R}^d),
\;\;\;
({\mathcal S}_1 ^{1})' (\mathbb{R}^d) =  \bigcup _{m  \in \cap \cM _{w_s}} M _{1/m} ^{p,q} (\mathbb{R}^d),
$$
where $ w_s $ is given by \eqref{eqc1}, see \cite{Toft-2017-a}.

In the context of quasianalytic Gelfand-Shilov spaces, we recall (a special case of) \cite[Theorem 3.9]{Toft-2012}:
Let $s,t> 1/2$ and set
$$
w_h (x, \omega) \equiv e^{h (|x|^{1/t} + |\omega|^{1/s})}, \;\;\; h>0, \; x,\omega \in \mathbb{R}^d.
$$
Then
$$
\Sigma_t ^s (\mathbb{R}^d)=  \bigcap _{h>  0} M _{w_{h}} ^{p,q} (\mathbb{R}^d),\;\;\;
(\Sigma_t ^s)' (\mathbb{R}^d) =  \bigcup _{h>0} M _{1/w_{h}} ^{p,q} (\mathbb{R}^d),
$$
$$
{\mathcal S}_t ^{s} (\mathbb{R}^d)
=  \bigcup _{h >  0} M _{w_{h}} ^{p,q} (\mathbb{R}^d),\;\;\;
({\mathcal S}_t ^{s})' (\mathbb{R}^d)
= \bigcap _{h > 0} M _{1/w_{h}} ^{p,q} (\mathbb{R}^d).
$$
\end{remark}

Modulation spaces include the following well-know function spaces:
\begin{enumerate}
\item $ M^2 (\mathbb{R}^d) = L^2 (\mathbb{R}^d),$  and $ M^2 _{t,0}(\mathbb{R}^d) = L^2 _t (\mathbb{R}^d);$
\item The Feichtinger algebra: $ M^1 (\mathbb{R}^d) = S_0 (\mathbb{R}^d);$
\item Sobolev spaces: $ M^2 _{0,s}(\mathbb{R}^d) = H^2 _s (\mathbb{R}^d) = \{ f \, | \,
\hat f (\omega) \langle \omega \rangle ^s \in  L^2 (\mathbb{R}^d)\};$
\item Shubin spaces: $ M^2 _{s}(\mathbb{R}^d) = L^2 _s (\mathbb{R}^d) \cap H^2 _s (\mathbb{R}^d) = Q_s (\mathbb{R}^d),$
cf. \cite{Shubin91}.
\end{enumerate}

\subsection{Convolution estimates for modulation spaces}

Different theorems concerning the convolution relation between modulation spaces can be found in the literature.
We recall the convolution estimates  given in \cite[Proposition 2.4]{CG02} and in \cite{TJPT2014}, which is sufficient for our purposes.

\begin{proposition}\label{mconvmp}
Let $m\in\cM_v$ defined on $\rdd$ and let $m_1(x)
= m(x,0) $ and $m_2(\omega ) = m(0,\omega )$,  the restrictions
to $\Ren\times\{0\}$ and  $\{0\}\times\Ren$, and likewise for $v$.
 Let $\nu (\omega )>0$ be  an arbitrary  weight function on $\Ren$ and \hfill\break\ni $1\leq
 p,q,r,s,t\leq\infty$. If
 $$\frac1p+\frac1q-1=\frac1r,\quad \,\, \text{ and } \,
 \quad\frac1t+\frac1{t'}=1\, ,$$
then
\begin{equation}\label{mconvm}
M^{p,st}_{m_1\otimes \nu}(\Ren)\ast  M^{q,st'}_{v_1\otimes
  v_2\nu^{-1}}(\Ren)\hookrightarrow M^{r,s}_m(\Ren)
\end{equation}
with  norm inequality  $\| f\ast h \|_{M^{r,s}_m}\lesssim
\|f\|_{M^{p,st}_{m_1\otimes \nu}}\|h\|_{ M^{q,st'}_{v_1\otimes
    v_2\nu^{-1}}}$.
\end{proposition}

When the weights in Proposition \ref{mconvmp} are chosen to be of the form $ \langle  x \rangle ^t
\langle \omega \rangle ^s$, sharper continuity properties can be proved.
For such results on multiplication and convolution in modulation spaces and in weighted Lebesgue spaces we
observe  the \emph{Young functional}:
\begin{equation} \label{R-functional}
\masfR (\mathrm{p}) = \masfR (p_0,p_1,p_2) \equiv 2-\frac 1{p_0}-\frac 1{p_1}-\frac 1{p_2},\qquad \mathrm{p}=
(p_0,p_1,p_2)\in [1,\infty]^3.
\end{equation}

When $ \masfR (\mathrm{p}) = 0, $ the Young inequality for convolution reads as
$$
\| f_1 * f_2 \|_{L^{p_0 ' }} \leq \| f_1  \|_{L^{p_1 }} \| f_2 \|_{L^{p_2}}, \;\;\;
f_j \in L^{p_j }(\mathbb{R}^d), \;\; j = 1,2.
$$

%

The following theorem is an extension of the Young inequality to the case of weighted Lebesgue spaces and
modulation spaces when $ 0\le \masfR (\mathrm{p}) \le 1/2$.

\begin{theorem} \label{mainconvolution}
Let $s_j,t_j \in \mathbb R$, $p_j,q_j \in [1,\infty] $, $j=0,1,2$.
Assume that $0\le \masfR (\mathrm{p}) \le 1/2$, $\masfR (\mathrm{q})\le 1$,
\begin{eqnarray}
0&\leq t_j+t_k,   & j,k=0,1,2,  \quad j\neq k, \label{lastineq2A}
\\
0 &\leq t_0+t_1 + t_2 - d \cdot  \masfR (\mathrm{p}), & \text{and} \label{lastineq2B}
\\
0 & \le s_0+s_1+s_2, & \label{lastineq2C}
\end{eqnarray}
with strict inequality in \eqref{lastineq2B} when $\masfR (\mathrm{p})>0$ and $t_j=d\cdot
\masfR (\mathrm{p}) $ for some $j=0,1,2$.

Then  $(f_1,f_2)\mapsto
f_1*f_2$ on $C_0^\infty (\mathbb{R}^d)$ extends uniquely to a continuous
map from
\begin{enumerate}
\item  $ L^{p_1} _{t_1}(\mathbb{R}^d) \times  L^{p_2} _{t_2}(\mathbb{R}^d)$ to
$L^{p_0'} _{-t_0}(\mathbb{R}^d)$;
\item
$M^{p_1,q_1} _{s_1,t_1}(\mathbb{R}^d) \times  M^{p_2,q_2} _{s_2,t_2}(\mathbb{R}^d)$ to
$M^{p_0',q_0'} _{-s_0,-t_0}(\mathbb{R}^d)$.
\end{enumerate}
\end{theorem}

For the proof we refer to \cite{TJPT2014}.
It is based on the detailed study of an auxiliary three-linear map
over carefully chosen regions in $ \mathbb{R}^d $ (see Subsections 3.1 and 3.2 in \cite{TJPT2014}).
This result extends multiplication and convolution properties obtained in \cite{PTT2}.
Moreover, the sufficient conditions from Theorem \ref{mainconvolution} are also necessary in the following sense.

\begin{theorem}\label{otpmimality1}
Let $p_j,q_j\in [1,\infty ]$ and $s_j,t_j\in \mathbb R$, $j=0,1,2$. Assume that at
least one of the following statements hold true:
\begin{enumerate}
\item the map $(f_1,f_2)\mapsto f_1*f_2$ on $C_0^\infty (\mathbb{R}^d)$ is continuously
extendable to a map from $L^{p_1}_{t_1}(\mathbb{R}^d)\times
L^{p_2}_{t_2}(\mathbb{R}^d)$ to $L^{p_0'}_{-t_0}(\mathbb{R}^d)$;

\vrum

\item the map $(f_1,f_2)\mapsto f_1*f_2$ on $C_0^\infty (\mathbb{R}^d)$ is continuously
extendable to a map from $M^{p_1,q_1}_{s_1,t_1}(\mathbb{R}^d)\times
M^{p_2,q_2}_{s_2,t_2}(\mathbb{R}^d)$ to $M^{p_0',q_0'}
_{-s_0,-t_0}(\mathbb{R}^d)$;

\end{enumerate}
Then \eqref{lastineq2A} and \eqref{lastineq2B} hold true.
\end{theorem}

\section{Localization operators} \label{localization}

We refer to \cite{CG02, CPRT1, CPRT2, Teof2016} for the continuity properties of localization operators on modulation spaces,
and here we give a reformulation of such results by using the Grossmann-Royer transform
instead of the cross-Wigner distribution. Furthermore, we use the Grossmann-Royer transform  to
define localization operators and to show that such operators are Weyl pseudodifferential operators.
For our purposes the duality between
$\mathcal{S} ^{(1)}(\mathbb{R}^d)$ and $ (\mathcal{S} ^{(1)})'(\mathbb{R}^d)$
will suffice, and we use it here for the simplicity and for the clarity of exposition.

\begin{definition} \label{locopdef}
Let $ f \in  {\mathcal S}^{( 1)} (\mathbb{R}^d)$.
The {\em localization operator} $A_a ^{\varphi _1, \varphi _2} $ with {\em symbol}
$a  \in  {\mathcal S} ^{(1)'} (\mathbb{R}^{2d} )$ and
{\em windows} $\varphi _1, \varphi _2 \in {\mathcal S} ^{(1)} (\mathbb{R}^d ) $
is given by
\begin{equation} \label{locopGRT}
A_a ^{\varphi _1, \varphi _2} f(t)=\int_{\mathbb{R}^{2d} } a (x,\omega ) R_{\check{\varphi _1}} f (\frac{x}{2},\frac{\omega}{2})
R (\check{\varphi _2} (t)) (\frac{x}{2},\frac{\omega}{2})
\, dx d\omega.
\end{equation}
\end{definition}

In the weak sense,
\begin{eqnarray} \label{locopGRTweak}
  \langle A_a ^{\varphi _1, \varphi _2} f,g \rangle & = &\langle  a (x,\omega ) R_{\check{\varphi_1}}f (\frac{x}{2},\frac{\omega}{2}),
  R_{\check{\varphi_2}}g (\frac{x}{2},\frac{\omega}{2})  \rangle \\
   & = & \langle  a (x,\omega ), R_{\check{\varphi_1}}f (\frac{x}{2}\frac{\omega}{2}), R_{\check{\varphi_2}}g (\frac{x}{2},\frac{\omega}{2})
  \rangle,
  \quad f,g\in  {\mathcal S} ^{(1)} (\mathbb{R}^d ),
\end{eqnarray}
so that $ A_a ^{\varphi _1, \varphi _2} $ is well-defined continuous operator from
$ {\mathcal S} ^{(1)} (\mathbb{R}^d ) $ to $({\mathcal S} ^{(1)})' (\mathbb{R}^d ) $, cf. Proposition \ref{svojstva}.

\begin{lemma} \label{locopsame}
Let there be given $ f\varphi _1, \varphi _2 \in  {\mathcal S}^{( 1)} (\mathbb{R}^d)$ and
$a  \in  {\mathcal S} ^{(1)'} (\mathbb{R}^{2d} )$.
Then $ A_a ^{\varphi _1, \varphi _2} $ given by \eqref{locopGRT} coincides
with the  usual localization operator  $\tilde{A}_a ^{\varphi _1, \varphi _2} $ defined by the short-time Fourier transform:
\begin{equation*}
\tilde{A}_a ^{\varphi _1, \varphi _2}  f(t)=\int_{\mathbb{R}^{2d} } a (x,\omega ) V_{\varphi _1} f (x,\omega ) M_\omega T_x \varphi _2 (t)
\, dx d\omega.
\end{equation*}
\end{lemma}
\par

\begin{proof}
From Lemma \ref{GRandRelatives} it follows that
\begin{eqnarray*}
\tilde{A}_a ^{\varphi _1, \varphi _2}  f(t) & = &
\int_{\mathbb{R}^{2d} } a (x,\omega ) e^{-\pi i x \omega} R_{\check{\varphi _1}} f (x,\omega ) M_\omega T_x \varphi _2 (t)
\, dx d\omega \\
& = &  \int_{\mathbb{R}^{2d} } a (x,\omega ) e^{-\pi i x \omega} R_{\check{\varphi _1}} f (x,\omega )
  e^{\pi i x \omega} R (\check{\varphi _2} (t)) (\frac{x}{2},\frac{\omega}{2})
\, dx d\omega \\
 & = &  A_a ^{\varphi _1, \varphi _2} f(t),
\end{eqnarray*}
since
\begin{eqnarray*}
R (\check{\varphi _2} (t)) (\frac{x}{2},\frac{\omega}{2})
& = & e^{4 \pi i \frac{\omega}{2}(t-\frac{x}{2})} \check{\varphi _2} (2 \frac{x}{2} - t) \\
& = & e^{-  \pi i \omega x} e^{2 \pi i \omega t} \varphi _2 (t-x) \\
& = & e^{- \pi i \omega x} M_\omega T_x \varphi _2 (t).
\end{eqnarray*}
and the lemma is proved.
\qed
\end{proof}

\par

Next we show that localization operators can be represented as  Weyl pseudodifferential operators.

Recall, if  $\sigma \in {\mathcal S}^{( 1)} (\mathbb{R}^{2d})$ then the
Weyl pseudodifferential operator $ L_\sigma$ is defined as the oscillatory integral:
$$
L_\sigma f (x) = \iint \sigma (\frac{x+y}{2}, \omega) f(y) e^{2\pi (x-y)\cdot \omega} dy d\omega, \;\;\;
f \in {\mathcal S}^{( 1)} (\mathbb{R}^{2d}).
$$
It extends to each $\sigma \in {\mathcal S}^{( 1)'} (\mathbb{R}^{2d})$, and then
$L_\sigma$ is continuous from $ {\mathcal S}^{( 1)} (\mathbb{R}^{2d}) $ to
$ {\mathcal S}^{( 1)'} (\mathbb{R}^{2d})$, and $\sigma $ is called {\em the Weyl symbol} of the  pseudodifferential operator $L_\sigma$.

\par

\begin{lemma} \label{Weyl psido char}
Let  $L_\sigma$ be the Weyl pseudodifferential operator with the Weyl symbol
$\sigma \in {\mathcal S}^{( 1)'} (\mathbb{R}^{2d})$. Then we have
$$
L_\sigma f (t) = 2^d \int_{\mathbb{R}^{2d}}  \sigma (x,\omega ) (R f(t)) (x,\omega ) dx d\omega, \;\;\; t \in \mathbb{R}^d,
$$
or, in the weak sense
\begin{equation} \label{WeylpsidoGRT}
 \langle L_\sigma f,g \rangle = 2^d \langle \sigma, R_{f} g \rangle,
 \quad\quad f,g\in {\mathcal S}^{( 1)} (\mathbb{R}^d).
\end{equation}
\end{lemma}

\begin{proof}
The lemma is the same as \cite[Proposition 40]{deGosson2017}. We give here a different proof.
In fact we only use Fubini's theorem and the change of variables $ 2x - y \mapsto t  $:
\begin{eqnarray*}
2^d \langle \sigma, R_{f} g \rangle
& = & 2^d \int \int \int \sigma (x,\omega ) e^{-4 \pi i \omega(y-x)}
\overline{g} (2x -y ) f(y) dy dx d\omega \\
& = & 2^d \int \int \left ( \int \sigma (x,\omega ) e^{4 \pi i \omega x}
\overline{g} (2x -y ) dx \right ) e^{-4 \pi i \omega y} f(y) dy d\omega \\
& = &  \int \int \int \sigma (\frac{t+y}{2},\omega ) e^{2 \pi i \omega (t+y)}
\overline{g} (t ) dt e^{-4 \pi i \omega y} f(y) dy d\omega \\
& = & \int \int \int \sigma (\frac{t+y}{2},\omega ) e^{2 \pi i \omega (t-y)}
f(y) \overline{g} (t ) dt  dy d\omega,
\end{eqnarray*}
therefore $ 2^d \langle \sigma, R_{f} g \rangle  =   \langle L_\sigma f,g \rangle.$
\qed
\end{proof}

Lemmas \ref{GRandRelatives} and \ref{Weyl psido char} imply the well-known formula:
$ \langle L_\sigma f,g \rangle = \langle \sigma, W (g,f) \rangle,$ cf. \cite{folland89, Shubin91, Wong1998}.

\par

Next we establish the so called Weyl connection, which shows that the set of
localization operators is a  subclass of the set of Weyl operators. Although the same result
can be found elsewhere (\cite{folland89, BCG02, Teof2016}),
it is given here in order to be self-contained. The proof is
based on  kernel theorem for Gelfand-Shilov spaces, and direct calculation.

\begin{lemma} \label{Weyl connection lemma}
If $ a \in  {\mathcal S}^{( 1)'} (\mathbb{R}^{2d})$ and
$\varphi _1, \varphi _2 \in {\mathcal S} ^{(1)} (\mathbb{R}^d ) $, then
the localization operator $A_a ^{\varphi _1, \varphi _2}  $
is Weyl pseudodifferential operator with the Weyl symbol
$ \sigma = 2^{-d} a\ast R_{\varphi_1} \varphi_2 $, in other words,
\begin{equation} \label{Weyl connection}
A_a ^{\varphi _1, \varphi _2}  = 2^{-d} L_{a\ast R_{\varphi_1} \varphi_2}.
\end{equation}
\end{lemma}

\begin{proof}
By the kernel Theorem \ref{kernelteorema}
it follows that for any
linear and continuous operator $T$ from $ {\mathcal S}^{( 1)} (\mathbb{R}^{2d}) $ to
$ {\mathcal S}^{( 1)'} (\mathbb{R}^{2d})$, there exists a uniquely determined
$ k \in {\mathcal S}^{( 1)'} (\mathbb{R}^{2d})$ such that
$$
\langle Tf, g \rangle = \langle k, g \otimes \overline{f} \rangle, \;\;\;
f,g \in {\mathcal S}^{( 1)} (\mathbb{R}^{2d}),
$$
see also \cite{LCPT, Teof2016, TKNN2012}.

We first calculate the kernel of $A_a ^{\varphi _1, \varphi _2}  $, and than show that it coincides to the kernel of $ L_{\sigma} $ when
$ \sigma = 2^{-d} a\ast R_{\varphi_1} \varphi_2$.

From \eqref{locopGRTweak} and Proposition \ref{svojstva} {\em 2.} it follows:
\begin{multline*}
\langle \aaf f, g \rangle
 =  \langle  a (x,\omega ) R_{\check{\varphi_1}}f (\frac{x}{2},\frac{\omega}{2}),
  R_{\check{\varphi_2}}g (\frac{x}{2},\frac{\omega}{2})
  \rangle \\
 =  \langle  a (x,\omega ) R_{\check{\varphi_1}}f (\frac{x}{2},\frac{\omega}{2}),
  R_{\check{\varphi_2}}g (\frac{x}{2},\frac{\omega}{2})
  \rangle \\
= \iint_{\mathbb{R}^{2d}} a(x,\omega)
\left ( \int_{\mathbb{R}^{d}} f(y) \overline{R ( \check{\varphi _1} (y))} (\frac{x}{2},\frac{\omega}{2}) dy \right )
\left ( \int_{\mathbb{R}^{d}} \overline{g} (t) R ( \check{\varphi _2} (t)) (\frac{x}{2},\frac{\omega}{2}) dt \right ) dx d\omega
\\
= \int_{\mathbb{R}^{d}} \int_{\mathbb{R}^{d}} f(y) \overline{g} (t)
\left ( \iint_{\mathbb{R}^{2d}} a(x,\omega) \overline{R ( \check{\varphi _1} (y))} (\frac{x}{2},\frac{\omega}{2})
R ( \check{\varphi _2} (t)) (\frac{x}{2},\frac{\omega}{2}) dx d\omega \right )   dt  dy
\\
=
\langle k, g \otimes \overline{f} \rangle,
\end{multline*}
where
\begin{equation} \label{kernel}
k(t,y) =  \int_{\mathbb{R}^{2d}} a(x,\omega) \overline{R ( \check{\varphi _1 } (y))} (\frac{x}{2},\frac{\omega}{2})
R ( \check{\varphi _2} (t)) (\frac{x}{2},\frac{\omega}{2}) dx d\omega.
\end{equation}

Next, we calculate the kernel of $ L_{a\ast R_{\varphi_1} \varphi_2} $.
We use the covariance property of the Grossmann-Royer transform, Proposition \ref{svojstva} {\em 3.}, to obtain
\begin{multline*}
a \ast R_{ \varphi_1}  \varphi_2 (p,q)
\\[1ex]
=
\iint_{\mathbb{R}^{2d}}  a(x,\omega) R_{ \varphi_1}  \varphi_2 (p-x,q-\omega) dx d\omega
\\[1ex]
=
\iint_{\mathbb{R}^{2d}}  a(x,\omega) R_{T_x M_\omega  \varphi_1} T_x M_\omega  \varphi_2  (p,q) dx d\omega
\\[1ex]
=
\iint_{\mathbb{R}^{2d}}  a(x,\omega)
\left ( \int_{\mathbb{R}^{d}} e^{4\pi i q (t-p)} T_x M_\omega  \varphi_2( 2p-t)
\overline{ T_x M_\omega  \varphi_1}
(t) dt \right) dx d\omega.
\end{multline*}

Now,
\begin{multline*}
\langle  a\ast R_{ \varphi_1}  \varphi_2, R_f  g \rangle
\\
=
\iint_{\mathbb{R}^{4d}}   a(x,\omega)
\left ( \int_{\mathbb{R}^{d}} e^{4\pi i q (t-p)} T_x M_\omega  \varphi_2( 2p-t)
\overline{ T_x M_\omega  \varphi_1}
(t) dt \right)  dx d\omega
\\
\times \left (\int_{\mathbb{R}^{d}}
e^{-4\pi i q (s-p)}
\overline{g} (2p -s) f(s) ds \right ) dp dq
\\
=
\iint_{\mathbb{R}^{3d}}   a(x,\omega)
\left ( \int_{\mathbb{R}^{d}} e^{4\pi i q (t-s)} dq
\int_{\mathbb{R}^{d}} T_x M_\omega  \varphi_2( 2p-t)
\overline{ T_x M_\omega   \varphi_1} (t) dt \right)
\\
\times \left (\int_{\mathbb{R}^{d}}
\overline{g} (2p -s) f(s) ds \right ) dp   dx d\omega
\end{multline*}
\begin{multline*}
=
\iint_{\mathbb{R}^{3d}}   a(x,\omega)
\delta (t-s)
\int_{\mathbb{R}^{d}} T_x M_\omega  \varphi_2( 2p-t)
\overline{ T_x M_\omega  \varphi_1} (t) dt
\\
\times
\left (\int_{\mathbb{R}^{d}}
\overline{g} (2p -s) f(s) ds \right )  dp   dx d\omega
\\
=
\iint_{\mathbb{R}^{4d}}   a(x,\omega)
T_x M_\omega  \varphi_2( 2p-t)
\overline{ T_x M_\omega  \varphi_1} (t)
\overline{g} (2p -t) f(t) dt  dp  dx d\omega
\end{multline*}
\begin{multline*}
=
\iint_{\mathbb{R}^{4d}}   a(x,\omega)
T_x M_\omega  \varphi_2( 2p-t)
\overline{ T_x M_\omega  \varphi_1} (t)
\overline{g} (2p -t) f(t) dt   dp  dx d\omega
\\
=
2^{-d}  \iint_{\mathbb{R}^{4d}}   a(x,\omega)
T_x M_\omega  \varphi_2( s)
\overline{ T_x M_\omega  \varphi_1} (t)
\overline{g} (s) f(t) dt ds dx d\omega
\\=
2^{-d}  \iint_{\mathbb{R}^{4d}}   a(x,\omega)
R (\check{\varphi_2} ( s)) (\frac{x}{2},\frac{\omega}{2})
\overline{ R( (\check{\varphi_1} (s))} (\frac{x}{2},\frac{\omega}{2})
\overline{g} (s) f(t) dt ds dx d\omega,
\end{multline*}
where we deal with  the oscillatory integral as we did before, and use the change of variables
$ 2p -t \mapsto t$.

Therefore
$$
\langle L_{ a\ast R_{\varphi_1} \varphi_2 } f, g \rangle =
\langle k, g \otimes \overline{f} \rangle,
$$
where the $k$ is given by \eqref{kernel}. By the uniqueness of the kernel we conclude that
$ \aaf  = L_{a\ast R_{\varphi_1} \varphi_2 },$
and the proof is finished.
\qed
\end{proof}

Note that Lemma \ref{Weyl connection lemma} can be proved in quasianalytic case
by the same arguments. However, in this paper we do not need such an extension.
Notice also that in the literature the symbol $a$ in Lemma \ref{Weyl connection lemma} is called the
{\em anti-Wick symbol} of the  Weyl pseudodifferential  operator $L_{\sigma}$.

\par

Lemma \ref{Weyl connection lemma} describes localization operators
in terms of the convolution which is a smoothing operator. This implies
different boundedness results of localization operators even if $a$ is an ultradistribution.
In what follows we review some of these results.

\subsection{Continuity properties} \label{cont-prop}

In this subsection we recall the continuity properties from \cite{Teof2016} obtained by
using the relation between the Weyl pseudodifferential operators and localization operators,
Lemma \ref{Weyl connection lemma}, and convolution results for modulation spaces
from Theorem \ref{mainconvolution}.

We also use sharp continuity results from \cite{CN}.
There it is shown that the sufficient conditions for the continuity of the cross-Wigner distribution on
modulation spaces are also necessary (in the un-weighted case).
Related results can be found elsewhere, e.g.  in \cite{Toft2, To8, Teof2016}. In many situations such results overlap.
For example, Proposition 10 in \cite{Teof2018} coincides with certain sufficient conditions from \cite[Theorem 1.1]{CN}
when restricted to $ \masfR (\mathrm{p}) = 0,$ $t_0 = -t_1,$ and $ t_2 = |t_0|$.
For our purposes it is convenient to rewrite \cite[Theorem 1.1]{CN}
in terms of the Grossmann-Royer transform.

\begin{theorem} \label{thm-cross-Wigner}
Let there be given $ s \in \mathbb{R}$  and $ p_i, q_i, p, q \in [1,\infty],$ such that
\begin{equation} \label{uslov1}
p \leq p_i, q_i \leq q, \;\;\; i = 1,2
\end{equation}
and
\begin{equation} \label{uslov2}
\min \left \{ \frac{1}{p_1} + \frac{1}{p_2},   \frac{1}{q_1} + \frac{1}{q_2} \right \}
\geq
\frac{1}{p} + \frac{1}{q}.
\end{equation}
If $f,g \in \mathcal{S} (\mathbb{R}^{d}), $
then the map $ (f,g) \mapsto  R_g f $
extends to sesquilinear continuous map from
$ M^{p_1, q_1} _{|s|} (\mathbb{R}^{d}) \times  M^{p_2,q_2} _{s} (\mathbb{R}^{d})  $ to
$ {M}^{p,q} _{s,0} (\mathbb{R}^{2d}) $ and
\begin{equation} \label{cross-Wig-mod}
\|  R_g f \|_{ {M}^{p,q} _{s,0} } \lesssim \| f\|_{M^{p_1, q_1} _{|s|}}  \|g \|_{M^{p_2,q_2} _{s} }.
\end{equation}
Viceversa, if there exists a constant $C>0$ such that
$$
\|  R_g f \|_{ {M}^{p,q}} \lesssim \| f\|_{M^{p_1, q_1} }  \|g \|_{M^{p_2,q_2} }.
$$
then \eqref{uslov1} and \eqref{uslov2} must hold.
\end{theorem}

\begin{proof}
We refer to \cite[Section 3]{CN} for  the proof. It is given there in terms of the
cross-Wigner distribution, which is the same as the Grossmann-Royer transform, up to the constant factor $2^d$.
\qed
\end{proof}

Let $\sigma $ be the Weyl symbol of $ L_\sigma $.
By \cite[Theorem 14.5.2]{Gro}
if $ \sigma \in M^{\infty, 1} (\mathbb{R}^{2d}) $ then $L_\sigma $ is bounded on
$ M^{p,q} (\mathbb{R}^{d})$, $ 1\leq p,q \leq \infty $.
This result has a long history starting with the Calderon-Vaillancourt theorem on boundedness of
pseudodifferential operators with smooth and bounded symbols on $ L^2 (\mathbb{R}^{d})$, \cite{CalVai}.
It is extended by Sj{\"o}strand in \cite{Sjostrand1}
where $M^{\infty,1}$ is used as appropriate symbol class.
Sj{\"o}strand's results were thereafter extended in
\cite{Grochenig0,Gro,Grochenig1b,Toft2,To8,Toft2007a}.

\par

\begin{theorem} \label{conv-a-cross-Wigner}
Let the assumptions of Theorem \ref{mainconvolution} hold.
If $\varphi _j \in M^{p_j} _{t_j} (\mathbb{R}^{d}), $ $ j=1,2$,
and $ a \in M^{\infty, r} _{u,v} ( \mathbb{R}^{2d})$ where $ 1\leq r \leq p_0 $,
$u \geq t_0 $ and $ v \geq d \masfR (\mathrm{p}) $ with
$ v > d \masfR (\mathrm{p}) $ when $ \masfR (\mathrm{p}) > 0 $,
then $ A _a ^{\varphi_1, \varphi_2} $ is bounded on $ M^{p,q} (\mathbb{R}^{d}), $
for all $ 1\leq p,q\leq \infty $ and the operator norm satisfies the uniform estimate
$$
\|A _a ^{\varphi_1, \varphi_2}\|_{op} \lesssim  \|a\|_{M^{\infty, r} _{u,v}} \|\varphi_1\|_{M^{p_1} _{t_1}}
\|\varphi_2\|_ {M^{p_2} _{t_2}}.
$$
\end{theorem}

\begin{proof}
Let  $\varphi _j \in M^{p_j} _{t_j} (\mathbb{R}^{d}), $ $ j=1,2$.
Then, by Theorem \ref{thm-cross-Wigner} it follows that
$ R_{\varphi_1} \varphi_2 \in M^{1, p_0 '} _{-t_0,0} (\mathbb{R}^{2d}). $
This fact, together with Theorem \ref{mainconvolution} {\em (2)} implies that
$$
a *  R_{\varphi_1} \varphi_2  \in M^{\tilde p, 1} (\mathbb{R}^{2d}), \;\;\; \tilde p \geq 2,
$$
if the involved parameters fulfill the conditions of the theorem.
Concerning  the Lebesgue parameters it is easy to see that
$ \tilde p \geq 2 $ is equivalent to $ \masfR (\mathrm{p}) = \masfR (p, \infty, 1) \in [0, 1/2],$ and
that $ 1\leq r \leq p_0 $ is equivalent to $ \masfR (\mathrm{q}) = \masfR  (\infty, r , p_0 ') \leq 1$.
It is also straightforward to check that
the choice of the weight parameters $u$ and $v$ implies that
$ a * R_{\varphi_1} \varphi_2  \in M^{\tilde p, 1}  (\mathbb{R}^{2d})$, $ \tilde p \geq 2 $.

In particular, if $\tilde p= \infty$ then $ a *  R_{\varphi_1} \varphi_2 \in M^{\infty, 1}  (\mathbb{R}^{2d})$.
From \cite[Theorem 14.5.2]{Gro} (and Lemma \ref{Weyl connection lemma})
it follows that $ A _a ^{\varphi_1, \varphi_2} $ is bounded on
$ M^{p,q} (\mathbb{R}^{d}), $ $ 1\leq p,q\leq \infty $.

The operator norm estimate also follows from \cite[Theorem 14.5.2]{Gro}.
\qed
\end{proof}

\par

\begin{remark}
When $ p_1 = p_2 = 1$, $ r = p_0 = \infty $ and $ t_1=t_2 = - t_0 = s \geq 0 $,
$u = -s, $ $ v = 0 $  we recover the celebrated Cordero-Gr\"ochenig
Theorem, \cite[Theorem 3.2]{CG02}, in the case of polynomial weights, with the uniform estimate
$$
\|A _a ^{\varphi_1, \varphi_2}\|_{op} \lesssim  \|a\|_{M^{\infty} _{-s,0}} \|\varphi_1\|_{M^{1} _{s}}
\|\varphi_2\|_ {M^{1} _{s}}
$$
in our notation.

Another version of Theorem \ref{conv-a-cross-Wigner} with symbols from weighted modulation spaces can be obtained by using
\cite[Theorem 14.5.6]{Gro} instead.
We leave this to the reader as an exercise.
\end{remark}

\subsection{Schatten-von Neumann properties} \label{schvonneumann}

In this subsection we recall the compactness properties from \cite{Teof2016}.

References to the proof of the following  well known theorem can be found in \cite{CG02}.

\begin{theorem} \label{sigma-Schatten}
Let $\sigma $ be the Weyl symbol of $ L_\sigma $.
\begin{enumerate}
\item If $ \sigma \in M^1 (\mathbb{R}^{2d})$  then $ \| L_\sigma \|_{S_1} \lesssim \| \sigma \|_{M^1}.$
\item If $ \sigma \in M^p (\mathbb{R}^{2d})$, $ 1\leq p \leq 2 $,
then $ \| L_\sigma \|_{S_p} \lesssim \| \sigma \|_{M^p}.$
\item If $ \sigma \in M^{p,p'} (\mathbb{R}^{2d})$, $ 2\leq p \leq \infty $,
then $ \| L_\sigma \|_{S_p} \lesssim \| \sigma \|_{M^{p,p'}}.$
\end{enumerate}
\end{theorem}

\par

The Schatten-von Neumann properties in the following Theorem are formulated in the spirit of \cite{CG02},
see also \cite{Toft2,To8}. Note that more general weights are considered in \cite{Toft2007a,Toft2007b}, leading
to different type of results. We also refer to \cite{FG2006} for compactness properties obtained by a different approach.

\begin{theorem} \label{Schatten}
Let $ \masfR (\mathrm{p}) $ be the Young functional given by \eqref{R-functional},
$s \geq 0$ and $ t \geq d \masfR (\mathrm{p})$ with
$ t > d \masfR (\mathrm{p})$ when $  \masfR (\mathrm{p}) >0$.
\begin{enumerate}
\item If $ 1 \leq p \leq 2 $ and $ p \leq r \leq  2p/(2-p) $
then the mapping $ (a,\varphi_1,\varphi_2) \mapsto A _a ^{\varphi_1, \varphi_2} $ is bounded from
$ M^{r, q} _{-s,t} \times M^{1} _s \times M^{p} _s $,  into $ S_p$, that is
$$
\| A _a ^{\varphi_1, \varphi_2} \|_{S_p} \lesssim \| a \|_{M^{r,q} _{-s,t}}
\| \varphi_1 \|_{M^{1} _s}
\| \varphi_2 \|_{M^{p} _s}.
$$
\item
If $ 2 \leq p \leq \infty $  and $ p  \leq r $
then the mapping $ (a,\varphi_1,\varphi_2) \mapsto A _a ^{\varphi_1, \varphi_2} $ is bounded from
$ M^{r, q} _{-s,t} \times M^{1} _s \times M^{p'} _s $, into $ S_p$, that is
$$
\| A _a ^{\varphi_1, \varphi_2} \|_{S_p} \lesssim \| a \|_{M^{r,q} _{-s,t}}
\| \varphi_1 \|_{M^{1} _s}
\| \varphi_2 \|_{M^{p'} _s}.
$$
\end{enumerate}
\end{theorem}

\begin{proof} {\em 1.}
By Theorem \ref{thm-cross-Wigner} it follows that
$ R_{\varphi_1 } \varphi_2 \in M^{1, p_w} _{-t_0,0} (\mathbb{R}^{2d}), $
with $ t_0 \geq -s $ and $ p_w \in [2p/(p+2), p].$ Therefore
$ R_{\varphi_1 } \varphi_2   \in M^{1, p} _{s,0} (\mathbb{R}^{2d}). $

To apply Theorem \ref{mainconvolution} we notice that
for $ \masfR (\mathrm{p}) = (p', r, 1)$
the condition $\frac{1}{2} \geq  \masfR (\mathrm{p}) \geq 0$ is equivalent to
$ p \leq r \leq  2p/(2-p) $, and that for
$ \masfR (\mathrm{p}) = (p', q, p)$
the condition $ \masfR (\mathrm{q}) \leq 1$ is equivalent to
$ 1 \leq q $.
Now, Theorem \ref{mainconvolution} implies that
$ a *  R_{\varphi_1 } \varphi_2   \in M^{p} (\mathbb{R}^{2d})$,
and the result follows from Theorem \ref{sigma-Schatten} {\em 2.}

{\em 2.} By Theorem \ref{thm-cross-Wigner}  it follows that
$ R_{\varphi_1 } \varphi_2 \in M^{1, p_w} _{-t_0,0} (\mathbb{R}^{2d}), $
with $ t_0 \geq -s $ and $ p_w \in [p', 2p'/(p'+2), p].$ Therefore
$ R_{\varphi_1 } \varphi_2 \in M^{1, p'} _{s,0} (\mathbb{R}^{2d}). $

Now, for $ 2 \leq p \leq \infty $, $\frac{1}{2} \geq  \masfR (\mathrm{p}) \geq 0$ is equivalent to  $ p  \leq r $,
and  for $ \masfR (\mathrm{q}) = (p, q, p')$
the condition $ \masfR (\mathrm{q}) \leq 1$ is equivalent to $ 1 \leq q $.
The statement follows from
Theorem \ref{mainconvolution} {\em 2.} and Theorem \ref{sigma-Schatten} {\em 3.}
\qed
\end{proof}

We remark that we corrected a typo from the formulation of \cite[Theorem 3.9]{Teof2016}
also note that a particular choice: $ r = p $, $ q = \infty $ and $ t= 0 $ recovers
\cite[Theorem 3.4]{CG02}.

In the next theorem we give some necessary conditions. The proof follow from the proofs of Theorems 4.3 and 4.4 in
\cite{CG02} and is therefore omitted.

\begin{theorem} \label{necessary}
Let the assumptions of Theorem \ref{mainconvolution} hold
and let $ a \in \mathcal{S}' (\mathbb{R}^{2d}).$
\begin{enumerate}
\item
If there exists a constant $ C = C(a) > 0 $ depending only on the symbol $a$ such that
$$
\|A _a ^{\varphi_1, \varphi_2}\|_{S_\infty} \leq C \|\varphi_1\|_{M^{p_1} _{t_1}}
\|\varphi_2\|_ {M^{p_2} _{t_2}},
$$
for all $\varphi _1, \varphi _2  \in  \mathcal{S} (\mathbb{R}^{d}), $
then $a \in M^{\infty, r} _{u,v} (\mathbb{R}^{2d})$ where $ 1\leq r \leq p_0 $,
$u \geq t_0 $ and $ v \geq d \masfR (\mathrm{p}) $ with
$ v > d \masfR (\mathrm{p}) $ when $ \masfR (\mathrm{p}) > 0 $.
\item
If there exists a constant $ C = C(a) > 0 $ depending only on the symbol $a$ such that
$$
\| A _a ^{\varphi_1, \varphi_2} \|_{S_2} \leq C
\| \varphi_1 \|_{M^{1} }
\| \varphi_2 \|_{M^{1} }
$$
for all $\varphi _1, \varphi _2  \in  \mathcal{S} (\mathbb{R}^{d}), $
then $a \in M^{2, \infty}  (\mathbb{R}^{2d})$.
\end{enumerate}
\end{theorem}

We finish the section with results from \cite{CPRT1} related to the weights which may have exponential growth.
To that end we use the boundedness result Theorem \ref{thm-cross-Wigner} formulated in terms of such weights.
The proof is a modified version of the proof of \cite[Proposition 2.5]{CG02} and therefore omitted.

\begin{lemma}\label{wigestp} Let $w_s,\tau_s$ be the weights defined in \eqref{eqc1} and \eqref{eqc2}. If $1\leq p\leq\infty$, $s\geq 0$, $\f_1\in M^{1}_{w_s}(\Ren)$ and
$\f_2\in M^{p}_{w_s}(\Ren)$, then  $R_{\f_1} \f_2\in
M^{1,p}_{\tau_s}(\Renn)$, and
\begin{equation}\label{wigest}
\| R_{\f_1} \f_2\|_{M^{1,p}_{\tau_s}}\lesssim
\|\f_1\|_{M^{1}_{w_s}}\|\f_2\|_{ M^{p}_{w_s}}.
\end{equation}
\end{lemma}

\begin{theorem}\label{tempbound}
Let  $A _a ^{\varphi_1, \varphi_2}$ be the localization operator with the symbol $a$ and windows $\f_1$ and $\f_2$.
\begin{enumerate}
\item If $s\geq 0$, $a\in  M^{\infty}_{1/\tau_s}(\Renn)$, and $\f_1,\f_2\in
M^{1}_{w_s}(\Ren)$, then $\gaw$ is bounded  on  $M^{p,q}(\Ren)$ for all
$1\leq p,q\leq\infty$, and the operator norm satisfies  the
uniform estimate
$$\|\gaw\|_{op}\lesssim  \|a\|_{M^{\infty}_{1/\tau_s}}\|\f_1\|_
{M^{1}_{w_s}}\|\f_2\|_{ M^{1}_{w_s}}.$$
\item
If $1 \leq p  \leq 2$, then  the mapping $(a,\f _1, \f _2) \mapsto
A _a ^{\varphi_1, \varphi_2} $ is  bounded  from
$M^{p,\infty }_{1/\tau _s}(\rdd ) \times M^1_{w_s} (\rd )\times
M^p_{w_s} (\rd )$ into $S_p$, in other words,
$$
\|\gaw\|_{S_p}\lesssim
\|a\|_{M^{p,\infty}_{1/\tau _s}}\|\f_1\|_{M^1_{w_s}}\|\f_2\|_{M^p_{w_s}}\,.
$$
\item  If $2 \leq p  \leq \infty$, then the mapping $(a,\f _1, \f _2) \mapsto
A _a ^{\varphi_1, \varphi_2} $ is  bounded  from
$M^{p,\infty }_{1/\tau _s} \times M^1_{w_s}\times M^{p'}_{w_s}$ into
$S_p$, and
$$\|\gaw\|_{S_p}\lesssim
\|a\|_{M^{p,\infty}_{1/\tau _s}}\|\f_1\|_{M^1_{w_s}}\|\f_2\|_{M^{p'}_{w_s}}\,
.$$
\end{enumerate}
\end{theorem}
\begin{proof}
{\em 1.} We use the convolution relation \eqref{mconvm} to show that the Weyl
symbol $a \ast  R_{\f_1} \f_2$ of $\gaw $ is in $M^{\infty ,1}$.
 If $\f_1,\f_2\in  M^{1}_{w_s}(\Ren)$, then by  \eqref{wigest}, we
 have $ R_{\f_1} \f_2\in M^{1}_{\tau_s}(\Renn)$.
Applying  Proposition \ref{mconvmp} in the form $M^{\infty} _{1/\tau
  _{s}} \ast M^1_{\tau _s} \subseteq M^{\infty , 1}$,  we obtain
$\sigma=a\ast  R_{\f_1} \f_2\in M^{\infty, 1}$. The result now
follows from  Theorem \ref{sigma-Schatten}  {\it a}).

Similarly, the proof of {\em 2.} and {\em 3.} is based on results of  Proposition \ref{mconvmp} and Theorem \ref{sigma-Schatten}, items b) - d).
\end{proof}
For the sake of completeness, we state the necessary boundedness result, which follows by
straightforward modifications of \cite[Theorem 4.3]{CG02}.

\begin{theorem}
\label{necfr}
Let $a \in \cS^{(1) '} (\Renn )$ and $s\geq 0$.  If there exists a constant
$C=C(a)>0$ depending only on $a$ such that
\begin{equation}
  \label{eqr3}
  \|\gaw\|_{S_\infty}\leq  C\, \|\f_1\|_{ M^1_{w_s}}\|\f_2\|_{
    M^1_{w_s}}
\end{equation}
 for all $ \f _1, \f _2 \in \cS^{(1) } (\Ren   )$,  then $a\in M^\infty _{1/
   \tau_s}$.
\end{theorem}

Note that a trace-class result for certain quasianalytic distributions
(based on the heat kernel and parametrix techniques)  is given in \cite{CPRT1}.

\section{Further extensions} \label{extensions}

The results of previous sections can be extended in several directions.
We mention here some references to product formulas,
Shubin type pseudodifferential operators, multilinear localization operators,
and generalizations to quasi-Banach spaces and quasianalytic spaces of test functions and their dual spaces (of Fourier ultra-hyperfunctions). The list of topics and references is certainly incomplete and we leave it to the reader to accomplish it according to his/hers preferences.

The product of two localization operators can be written as a localization operator in very few cases. Exact formulas hold only in special cases, e.g. when the windows are Gaussians, \cite{Wong02}. Therefore, a symbolic calculus is developed in \cite{CorGro2006}
where such product is written as a sum of localization operators plus a remainder term expressed in a Weyl operator form. This is done in linear case in the framework of the symbols which may have subexponential growth.

The multilinear localization operators were first introduced in \cite{CKasso}
and their continuity properties are formulated in terms of modulation spaces. The
key point is the interpretation of these operators as multilinear Kohn-Nirenberg pseudodifferential operators.
The  multilinear pseudodifferential operators were already studied in the context of modulation spaces in \cite{BGHKasso,BKasso2004, BKasso2006},
see also a more recent contribution \cite{MOPf} where such approach is strengthened and applied to the bilinear and trilinear Hilbert transforms.

To deal with multilinear localization operators,
instead of standard modulation spaces $M^{p,q}$ observed in \cite{CKasso},
continuity properties in \cite{Teof2018,Teof2019} are formulated in terms of a modified version of modulation spaces denoted by $\mathcal{M}^{p,q} _{s,t}$, and given as follows.

By a slight abuse of the notation (as it is done in e.g. \cite{MOPf}),
$ \vec{f} $  denotes both the vector $ \vec{f} = (f_1,f_2, \dots, f_n) $ and the tensor product
$ \vec{f} = f_1 \otimes f_2 \otimes \dots \otimes f_n $. This will not cause confusion, since
the meaning of $ \vec{f} $  will be clear from the context.

For example, if $t=(t_1, t_2, \dots, t_n) $,  and $F_j = F_j (t_j),$ $t_j \in \mathbb{R}^d $,   $ j = 1,2,\dots,n$,
then
\begin{equation} \label{product}
\prod_{j=1} ^{n} F_j (t_j)= F_1 (t_1) \cdot  F_2 (t_2) \cdot \dots \cdot F_n (t_n) = F_1 (t_1) \otimes  F_2 (t_2) \otimes \dots \otimes F_n (t_n) = \vec{F} (t).
\end{equation}

To give an interpretation of multilinear operators in the weak sense we note that, when
$ \vec{f} = (f_1,f_2, \dots, f_n) $ and $ \vec{\varphi} = (\varphi_1, \varphi_2, \dots, \varphi_n) $,
$ f_j, \varphi _j   \in {\mathcal S} ^{(1)} (\mathbb{R}^d )$, $ j = 1,2,\dots, n,$
we put
\begin{equation} \label{STFTtensor-n}
R_{\vec{\varphi}} \vec{f} (x,\omega)
= \int_{\mathbb{R}^{nd}} \vec{f} (2x - t) \prod_{j=1} ^n e^{4\pi i \omega_j (t_j - x_j)}\varphi_j (t_j) dt,
\end{equation}
see also \eqref{product} for the notation.

\par

According to \eqref{STFTtensor-n}, $ \mathcal{M}^{p,q}_{s,t} (\mathbb{R}^{nd}) $ denotes the the set of
$ \vec{f} = (f_1,f_2, \dots, f_n), $ $ f_j \in {\mathcal S}' (\mathbb{R}^{d})$,  $ j = 1,2,\dots,n$, such that
$$
\| \vec{f} \|_{\mathcal{M}^{p,q}_{s,t}} \equiv \left  ( \int _{\mathbb{R}^{nd}} \left ( \int _{\mathbb{R}^{nd}}
| V_{\vec{\varphi}} \vec{f} (x,\omega)
\langle  x \rangle ^t
\langle \omega \rangle ^s|^p\, dx  \right )^{q/p}d\omega  \right )^{1/q}<\infty,
$$
where $ \vec{\varphi} = (\varphi_1, \varphi_2, \dots, \varphi_n) $,
$  \varphi _j   \in {\mathcal S}  (\mathbb{R}^d ) \setminus 0$, $ j = 1,2,\dots, n,$ is a given window function.

\par

The kernel theorem for $  {\mathcal S} (\mathbb{R}^{d})$ and $ {\mathcal S}' (\mathbb{R}^{d})$
(see \cite{Trev}) implies that there is an isomorphism between $ \mathcal{M}^{p,q}_{s,t} (\mathbb{R}^{nd}) $
and $M^{p,q}_{s,t}(\mathbb{R}^{nd})$ (which commutes with the operators from \eqref{trans-mod}). This allows us to identify
$ \vec{f} \in  \mathcal{M}^{p,q}_{s,t} (\mathbb{R}^{nd}) $ with (its isomorphic image) $ F \in M^{p,q}_{s,t}(\mathbb{R}^{nd})$ (and vice versa).

Next we give multilinear version of Definition \ref{locopdef}.

\begin{definition} \label{multlocopdef}
Let $f_j \in  {\mathcal S}(\mathbb{R}^d)$, $j=1,2,\dots,n$,
and $ \vec{f} = (f_1,f_2, \dots, f_n) $.
The {\em multilinear localization operator}\index{localization operator!multilinear} $\aaf $ with {\em symbol}
$a  \in  {\mathcal S} (\mathbb{R}^{2nd} )$ and {\em windows}
$$\varphi= (\varphi _1, \varphi _2, \dots, \varphi _n) \;\; \text{and} \;\;
\phi = (\phi_1, \phi_2, \dots, \phi _n), \;\;\; \varphi _j, \phi_j \in {\mathcal S}  (\mathbb{R}^d ), \;\;\; j = 1,2,\dots,n,
$$
is given by
$$
\aaf \vec{f} (t) =\int_{\mathbb{R}^{2nd} } a (x,\omega )
 \prod_{j=1} ^{n}\left (  R_{\check{\varphi _j}} f_j (\frac{x_j}{2},\frac{\omega_j}{2})
R (\check{\phi _j} (t)) (\frac{x_j}{2},\frac{\omega_j}{2}) \right )
\, dx d\omega,
$$
where
$x_j, \omega_j , t_j  \in \mathbb{R}^d, $ $ j=1,2,\dots, n,$ and
$ x = (x_1,x_2, \dots, x_n), $ $ \omega = (\omega_1 ,\omega_2 \dots, \omega_n ), $ $t = (t_1,t_2 \dots, t_n) $.
\end{definition}

Let $ \mathcal{R}$ denote the trace mapping that assigns to each function $F$
defined on  $ \mathbb{R}^{nd}$ a function defined on $ \mathbb{R}^{d}$ by the formula
$$
 \mathcal{R}: F \mapsto F \left | _{{t_1=t_2=\dots = t_n}} , \;\;\; t_j \in \mathbb{R}^{d}, \; j=1,2,\dots, n. \right .
$$
Then $  \mathcal{R} \aaf $ is the multilinear operator given in \cite[Definition 2.2]{CKasso}.

The approach  to multilinear localization operators related to  Weyl pseudodifferential operators is given in \cite{Teof2018,Teof2019}. Both Weyl and  Kohn-Nirenberg correspondences are particular cases of the
so-called $\tau-$pseudodifferential operators, $\tau \in [0,1]$ ($ \tau = 1/2$ gives Weyl operators, and $\tau = 0 $ we reveals Kohn-Nirenberg operators). We refer to \cite{Shubin91} for such operators,
and e.g. \cite{CDT, CNT, delgado, Toft-2017-b} for recent contributions in that context (see also the references given there).

The properties of such multilinear operators and their extension to Gelfand-Shilov spaces and their dual spaces
will be the subject of a separate contribution.

Furthermore, to obtain continuity properties  in the framework of quasianalytic Gelfand-Shilov spaces and corresponding dual spaces
analogous to those given in Section \ref{localization}, another techniques should be used, cf. \cite{CPRT2}.

Finally, we mention possible extension to quasi-Banach modulation spaces, when the Lebesgue parameters $p$ and $q$ are allowed to
take the values in $(0,1)$ as well. For such spaces and broad classes of pseudodifferential operators acting on them
we refer to \cite{Toft-2017-aa, Toft-2017-b}.

\begin{acknowledgement}
This research is supported by MPNTR of Serbia, projects no. 174024, and DS 028 (TIFMOFUS).
\end{acknowledgement}

\end{document}